\documentclass{amsart}

\usepackage{amsmath}
\usepackage{amsfonts}
\usepackage{amssymb}
\usepackage{hyperref}
\usepackage{subfig}
\usepackage{graphicx}
\usepackage{array}
\usepackage{indentfirst}
\usepackage{cite}
\usepackage{esint}
\usepackage{enumerate}

\newtheorem{theorem}{Theorem}[section]
\newtheorem{lemma}[theorem]{Lemma}
\newtheorem{corollary}[theorem]{Corollary}

\theoremstyle{definition}
\newtheorem{definition}[theorem]{Definition}

\theoremstyle{remark}
\newtheorem{remark}[theorem]{Remark}

\numberwithin{equation}{section}

\begin{document}

\title[Rectifiability of line defects]{Rectifiability of line defects in liquid crystals with variable degree of orientation}

\author{Onur Alper}
\address{Courant Institute of Mathematical Sciences, \newline \indent 251 Mercer Street, New York, NY 10012, USA}
\email{alper@cims.nyu.edu}

\subjclass[2000]{58E15, 58E20, 49Q20, 76A15}




\begin{abstract}
In \cite{AHL} Hardt, Lin and the author proved that 
the defect set of minimizers of the modified Ericksen energy for nematic liquid crystals
consists locally of a finite union of isolated points and H\"older continuous curves with finitely many crossings.
In this article, we show that each H\"older continuous curve in the defect set is of finite length. Hence, locally the defect set is rectifiable. 
For the most part, the proof follows the work of De Lellis, Marchese, Spadaro and Valtorta \cite{DLMSV} on harmonic $\mathcal{Q}$-valued maps closely.
The blow-up analysis in \cite{AHL} allows us to simplify the covering arguments in \cite{DLMSV} and locally estimate the length of line defects in a geometric fashion.
\end{abstract}

\maketitle

\section{Introduction}

\subsection{The Singularities of Energy-Minimizing Maps}
The structure of singularities of energy-minimizing harmonic maps has been studied for almost thirty years.
In \cite{HardtLin90} Hardt and Lin proved that the singular set of energy minimizing maps from $\mathbf{B}^4$ to $\mathbf{S}^2$ consists locally of
a finite union of isolated points and H\"older continuous curves with finitely many crossings.
An important ingredient in their proof was Reifenberg's Topological Disk Theorem, which was introduced in \cite{Reifenberg60}.
In \cite{Simon95} L. Simon proved that when the target is an analytic Riemannian manifold, the singular set of an energy minimizing maps on an $n$-dimensional domain is locally
the union of a finite, pairwise disjoint collection of $(n-3)$-rectifiable locally compact sets. 
In \cite{NV} Naber and Valtorta not only proved the rectifiability of singular sets of energy-minimizing and stationary harmonic maps into smooth Riemannian manifolds,
but they also introduced a variety of novel techniques, which they utilized to estimate the Minkowski content of singular sets.
In particular, they formulated new versions of the classical Reifenberg Theorem, involving the Jones $\beta_2$-number, originally introduced in \cite{Jones}, 
and yielding stronger results such as measure bounds.
Most recently, in \cite{DLMSV} De Lellis, Marchese, Spadaro and Valtorta combined the techniques in \cite{NV} with
a refined version of a monotonicity estimate introduced by G. Weiss in \cite{Weiss}, as well as a related frequency pinching estimate.
As a result, they proved that the singular set of energy-minimizing harmonic $\mathcal{Q}$-valued maps on an $n$-dimensional domain is countably $(n-2)$-rectifiable.
In fact, they also gave upper bounds for the $(n-2)$-dimensional Minkowski content of the subset of singular points with highest multiplicity.
See also \cite{KW} for earlier rectifiability results for the singular set of harmonic $\mathcal{Q}$-valued maps.

In \cite{AHL} Hardt, Lin and the author considered the structure of defect set of energy-minimizing harmonic maps into 
positively curved, single-sheeted cones over the projective plane $\mathbf{RP}^2$.
The motivation for studying maps into this singular target comes from the Ericksen model for nematic liquid crystals. 
In \cite{Ericksen91} Ericksen introduced an energy for director fields with variable degree of orientation, 
in order to allow finite-energy configurations with line defects, which are observed experimentally. 
In the {\emph{one-constant approximation}} regime \eqref{OneConstantApprox}, 
Lin related the minimizers of this energy to energy-minimizing maps into cones over $\mathbf{S}^2$ in \cite{Lin89}, \cite{Lin91}.
In \cite{HardtLin93} Hardt and Lin refined the dimension reduction result in \cite{Lin89}, \cite{Lin91} 
by classifying homogeneous minimizers of two variables, and introduced the modified Ericksen
energy for maps into positively curved, single-sheeted cones over $\mathbf{RP}^2$, for which energy-minimizing maps admit line defects as desired.
Building upon the classification result in \cite{HardtLin93}, 
it was proved in \cite{AHL} that under the assumption \eqref{OneConstantApprox}, the defect set of minimizers of the modified Ericksen energy consists locally of
a finite union of isolated points and H\"older continuous curves with finitely many crossings.

In this article we adopt the approach in \cite{DLMSV} to strengthen the main result in \cite{AHL}.
Assuming \eqref{OneConstantApprox} once again, we consider the minimizers of the modified Ericksen energy in a domain $\Omega$.
We prove that each H\"{o}lder continuous curve in the defect set of minimizers has locally finite length. 
Hence, combining this result with the above-mentioned structure theorem in \cite{AHL}, we arrive at the following conclusion: 
Under the assumption \eqref{OneConstantApprox}, locally the defect set is rectifiable.
Namely, the intersection of the defect set and any compact subset of $\Omega$ is the union of finitely many Lipschitz curves and a set of $\mathcal{H}^1$-measure zero.

\subsection{The Ericksen Model} \label{IntroModel}
We give a brief introduction to the Ericksen model in the context of equilibrium theory of liquid crystals.
Our main motivation here is to emphasize that the modified Ericksen model is the simplest equilibrium model admitting energy-minimizing configurations with line defects, 
which are experimentally observed.
We refer to \cite{Lin89} and the references therein for further details on the origins of various models we introduce below.

Let $\mu_x$ be a probability measure on $\mathbf{S}^2$ for the direction of a symmetric, elongated liquid crystal molecule at a given point $x$ 
in a spatial domain $\Omega \subset \mathbf{R}^3$.
The lack of polarity implies that $\int_{\mathbf{S}^2} p \, \mathrm{d} \mu_x(p) = 0$, 
while the anisotropy of the liquid crystal sample in $\Omega$ is captured by its second moment, $M_2(x) = \int_{\mathbf{S}^2} p \otimes p \, \mathrm{d} \mu_x(p)$.
In the uniaxial regime $M_2 - \frac{1}{3} id$ has two equal eigenvalues, and we can write:
$$
M_2 - \frac{1}{3} id = s \left [ (n \otimes n ) - \frac{1}{3} id \right ],
$$
where $| n | = 1$, $s \in \left [ -1/2, 1 \right ]$.

Assuming $s$ to be constant, the Oseen-Frank energy is defined as $\int_{\Omega} W(n)) \, \mathrm{d}x$, where:
\begin{equation*}
 W(n) = \kappa_1 | \mathrm{div} n |^2 + \kappa_2 | n \cdot \mathrm{curl} n |^2 + \kappa_3 | n \times \mathrm{curl} n |^2 
+ \left ( \kappa_2 + \kappa_4 \right ) \left [ \mathrm{tr} ( \nabla n )^2 - ( \mathrm{div}n )^2 \right ],
\end{equation*}
$\kappa_1$, $\kappa_2$, $\kappa_3 > 0$. 
The Oseen-Frank model is phenomenological, and $W(n)$ is the most general integrand depending on $n$ and $\nabla n$, while
satisfying at most quadratic dependence on $\nabla n$, as well as invariance under rotations and reflections of the director field $n$, cf. \cite[Section 2]{Lin89}.
Hardt, Kinderlehrer and Lin showed in \cite{HKL86} that the singular set of $n \, :  \, \Omega \to \mathbf{S}^2$ minimizing 
the Oseen-Frank energy has Hausdorff dimension strictly less than $1$.
As a result, the minimizers of Oseen-Frank energy cannot have line defects.

In \cite{Ericksen91} and \cite{Maddocks}, the Ericksen energy $\int_{\Omega} X(s,n)) \, \mathrm{d}x$ is introduced, where:
\begin{equation}
X(s,n) = s^2 W(n) + \kappa_5 | \nabla s|^2 + \kappa_6 | \nabla s \cdot n |^2 + \psi(s), \label{FullEricksen}
\end{equation}
and $\psi$ is a $C^2$-potential that satisfies: 
\begin{enumerate}[(i)]
 \item $\lim_{s \to -1/2} \psi(s) = + \infty$,
 \item $\lim_{s \to 1} \psi(s) = + \infty$,
 \item $\psi'(0) = 0$,
 \item $\psi$ has a minimum at some $s_* \in (0,1)$,  
\end{enumerate}
and serves to confine $s$ to the range $(-1/2,1)$. See also \cite[Fig.1]{Maddocks}. We refer to \cite{Ericksen91} and \cite{Maddocks} for 
a discussion of the Ericksen model in full generality.

In the {\emph{one-constant approximation}} regime, we assume: 
\begin{equation}
\kappa_1 = \kappa_2 = \kappa_3 = 1, \quad \kappa_4 = \kappa_6 = 0, \quad \kappa_5 = \kappa,
\label{OneConstantApprox}
\end{equation}
Under the assumption \eqref{OneConstantApprox}, the Ericksen energy reduces to:
\begin{equation}
\int_{\Omega} \left [ \kappa |\nabla s|^2 + s^2 | \nabla n |^2 + \psi(s)  \right ] \, \mathrm{d}x.
\label{Ericksen}
\end{equation}
In \cite{Lin89}, \cite{Lin91}, Lin proved the following result: there exists a map $u = \left ( u_1, u_2 \right ) \, : \, \mathbf{\Omega}\to \mathbf{C}_\kappa$ minimizing 
$$
\int_{\Omega} \left [ | \nabla u |^2 + \psi \left ( \kappa^{-1/2} | u | \right ) \right ] \, \mathrm{d}x, 
$$
with respect to Dirichlet boundary data $g \in H^{1/2} \left ( \Omega \right )$, where:
$$
\mathbf{C}_\kappa = \left \{ (z,y) \in \mathbf{R} \times \mathbf{R}^3 \, : \, z = \sqrt{\kappa-1} |y| \right \}, \quad \kappa > 1;
$$
the minimizer $u$ is locally H\"older continuous in $\Omega$; and the Hausdorff dimension of $u^{-1} \{ 0 \}$ is at most $1$.
(These existence and regularity results remain valid for a double-cone over $\mathbf{S}^2$, corresponding to $s \in ( -1/2, 1 )$, or in the case $\kappa <1$.
However, since minimizers can have wall defects in these cases, cf. \cite{AmbrosioVirga91}, \cite{HardtLin93}, we restrict our attention to the case $\kappa > 1$, $s \geq 0$.)
Away from the set $u^{-1} \{ 0 \}$, we can define the pair $s = \kappa^{-1/2} |u|$, $n = \sqrt{\kappa} |u|^{-1} u_2$, which minimizes \eqref{Ericksen}.
Moreover, the singular set of director field $n$, $\mathrm{sing}(n) = s^{-1} \{ 0 \} = u^{-1} \{ 0 \}$, which we call the zero set or the {\emph{defect set}}. 
Finally, it is standard to show that $u$, and consequently $s$ and $n$, are smooth away from $u^{-1} \{ 0 \}$.
In other words, the problem of analyzing the minimizers of \eqref{Ericksen}, where $\kappa > 1$ and $s \geq 0$, 
can be recast as the problem of analyzing energy-minimizing harmonic maps (in the presence of a lower order perturbation term) into the cone $\mathbf{C}_\kappa$. 

In \cite{HardtLin93} Hardt and Lin proved that when the target for energy-minimizing $u$ is $\mathbf{C}_\kappa$, $u^{-1} \{ 0 \}$ is in fact locally discrete. 
(In particular, energy-minimizing maps into $\mathbf{C}_\kappa$ do not admit line defects.)
They also introduced the modified Ericksen model by replacing the target $\mathbf{C}_\kappa$ with its projectivised version $\mathbf{D}_\kappa$. 
That is, for $[y] = \left \{ y, -y \right \}$, the sign equivalence class for $y \in \mathbf{R}^3$: 
$$
\mathbf{D}_\kappa = \left \{ \left ( z, [y] \right ) \, : \, (z,y) \in \mathbf{C}_\kappa \right \}, \quad \kappa > 1.
$$
While the results in \cite{Lin89}, \cite{Lin91} hold for the modified Ericksen model as well,
the modified Ericksen model has the additional remarkable feature that it admits minimizing configurations with line defects, cf. \cite{HardtLin93}, \cite{AHL}. 
See also \cite{BZ} for a comprehensive discussion on modeling liquid crystals with line fields and its relation to understanding defects in $Q$-tensor theory. 

Finally, we remark that in \cite{LinPoon} the existence and regularity of minimizers of the Ericksen energy without the {\emph{one-constant} approximation} \eqref{OneConstantApprox} are established.
However, it is an open problem whether the estimates on the Hausdorff dimension of the singular set proved under the assumption \eqref{OneConstantApprox} 
in \cite{Lin89}, \cite{Lin91} and \cite{HardtLin93} 
are valid for general material constants $\kappa_i$, $i=1,2,...,6$. 
We point out that this problem is analogous to the following question: does the singular set of Oseen-Frank minimizers consist locally of isolated points?
Once again, under the assumption \eqref{OneConstantApprox}, the Oseen-Frank energy reduces to the Dirichlet energy for maps into $\mathbf{S}^2$ or $\mathbf{RP}^2$. 
Therefore, the celebrated result of Schoen and Uhlenbeck \cite{SU82} on energy-minimizing harmonics maps, 
which crucially relies on the monotonicity of renormalized Dirichlet energy and Federer's dimension reduction principle, gives a positive answer to this question.
Without the assumption \eqref{OneConstantApprox}, 
discovering a monotone quantity for the minimizers of Oseen-Frank energy and constructing a dimension reduction argument is a challenging open problem.
(Nevertheless, we point out the result in \cite{HKL86} that 
the singular set of $n \, :  \, \Omega \to \mathbf{S}^2$ minimizing the Oseen-Frank energy has Hausdorff dimension strictly less than $1$.)
In short, the monotonicity-based techniques in \cite{SU82}, \cite{Simon95}, \cite{NV}, \cite{DLMSV} and \cite{AHL} are not applicable to the Oseen-Frank or Ericksen models 
without the assumption \eqref{OneConstantApprox} at this stage. Hence, in this article we assume \eqref{OneConstantApprox} hereafter.

\subsection{The Main Result} \label{IntroResult}
As in \cite{AHL}, for simplicity we consider energy-minimizing harmonic maps $u \, : \, \Omega \to \mathbf{D}_\kappa$, where $\Omega \subset \mathbf{R}^3$.
We analyze the local structure of $u^{-1} \{ 0 \}$ in full generality in the following theorem, which is a strengthened version of \cite[Theorem 5.2]{AHL}: 

\begin{theorem} \label{LocalFiniteLength}
Let $u \, : \, \Omega \to \mathbf{D}_\kappa$ be an energy-minimizing harmonic map, where $\Omega \subset \mathbf{R}^3$.
Then each point $b \in u^{-1} \{ 0 \}$ has an open neighborhood $\mathcal{O}$ so that
$ \left ( \mathcal{O} \backslash \{ b \} \right ) \cap u^{-1} \{ 0 \}$ consists of 
an even number of disjoint, embedded H\"older continuous arcs $\Gamma_1, \Gamma_2, ..., \Gamma_{2m}$,
joining $\{b\}$ to a point in $\partial \mathcal{O}$, where $2m \leq K_0 \left ( N \left ( b; 0^+ \right ) \right )$, 
and for $1 \leq i < j \leq 2m$, $r$ sufficiently small and $d_0 = d_0 \left ( N \left ( b; 0^+ \right ) \right )$,
\begin{equation*}
\mathrm{dist} \left ( \Gamma_i \cap \partial B_r(a) , \Gamma_j \cap \partial B_r (a) \right ) \geq \frac{1}{2} d_0 r.
\end{equation*}
Moreover, $\mathcal{H}^1 \left ( \Gamma_i \right ) < L = L \left ( N \left ( b; 0^+ \right ), \mathrm{diam}(\mathcal{O}) \right )$ for each $i = 1, 2, ..., 2m$. 
\end{theorem}

From Theorem \ref{LocalFiniteLength}, we easily obtain the main result of this article: 

\begin{theorem} \label{Strengthened}
Let $u \, : \, \Omega \to \mathbf{D}_\kappa$ be an energy-minimizing harmonic map and let $K$ be any compact subset of $\Omega \subset \mathbf{R}^3$. 
Then $u^{-1} \{ 0 \} \cap K$ consists of a finite union of isolated points and H\"older continuous curves of finite length and with finitely many crossings.
In particular, $u^{-1} \{ 0 \} \cap K$ is rectifiable, namely it is the union of finitely many Lipschitz curves and a set of $\mathcal{H}^1$-measure zero. 
\end{theorem}

Noting that the potential term $\psi$ in \eqref{FullEricksen} has $C^2$-dependence on $|u|$ only, and $\psi'(0)=0$, $\psi'(s)s \leq M s^2$ holds for a suitably large constant $M > 0$, 
depending on $\psi$.  
Consequently, we can generalize the variational formulas in Section \ref{VariationalSection}, as in \cite{Lin89}, \cite[Section3]{Lin91}.
More precisely, one can replace the classical Dirichlet energy, $D(x,r)$, in the variational formulas with:
\begin{equation}
\tilde{D}(x,r) = \int_{B_r(x)} \left [ \left | \nabla u \right |^2 + \psi' \left ( \kappa^{-1/2} |u| \right ) \kappa^{-1/2} |u| \right ] \, \mathrm{d}x.
\label{ModifiedDirichlet}
\end{equation}
Likewise, defining the height function, $H(x,r) = \int_{\partial B_r(x)} |u|^2 \, \mathrm{d}x$, the Almgren frequency function $N(x,r) = \frac{r D(r)}{H(r)}$,  can be replaced with:
\begin{equation}
 \tilde{N}(x,r) = e^{\Lambda r} \frac{r \tilde{D}(x,r)}{H(x,r)}, \label{GeneralizedFrequency}
\end{equation}
for a suitably large $\Lambda$, depending on the subcritical potential term $\psi$ and the constant $\kappa$ only.
By \cite[Lemma 7.1]{Lin91}, $\tilde{N}(x,r)$ is monotone increasing in $r$. 
In short, one can adapt the smoothed quantities introduce in Section \ref{VariationalSection} analogously.
Therefore, the conclusions we draw for energy-minimizing harmonic maps are also valid in the presence of a subcritical potential as described below the equation \eqref{FullEricksen}.
In particular, the conclusion of Theorem \ref{Strengthened} holds for the minimizing configurations of the modified Ericksen energy under the assumption \eqref{OneConstantApprox}.

\subsection{The Strategy of the Proof} \label{IntroProof}

We remark that the framework introduced in \cite{DLMSV} is extremely robust and makes very economical use of the properties of harmonic $\mathcal{Q}$-valued maps,
other than the variational structure captured in the identities in \cite[Proposition 3.1]{DLMSV}.
Similar ideas, as well as those borrowed from \cite{NV}, have recently found application in the work of Focardi and Spadaro \cite{FS} on the lower dimensional obstacle problem, 
which enjoys the same variational structure.  
In fact, the variational identities recalled in Lemma \ref{VarIdentities} are satisfied by any stationary map into a target that is a cone.
In a future work we will apply this framework to stationary harmonic maps into $k$-trees, as such maps naturally arise in problems concerning the optimal partition of domains, 
cf. \cite{CL07}, \cite{CL08}, \cite{CL10}.

We will use the basic compactness, regularity and dimension reduction results, which originally appeared in \cite{Lin89}, \cite{Lin91}, 
(cf. \cite[Lemma 2.7]{AHL}, \cite[Lemma 2.8]{AHL} and \cite[Lemma 2.11]{AHL} respectively.) 
Furthermore, a useful corollary of these ingredients is the unique continuation result for energy-minimizing maps into $\mathbf{D}_\kappa$ in Lemma \ref{UniqueContinuation}. 
Otherwise, Sections \ref{VariationalSection} and \ref{GeometricSection} rely heavily on variational ingredients, and therefore, follow \cite{DLMSV} very closely.
In particular, in Section \ref{VariationalSection} we define the smoothed quantities that play a central role in our analysis and
prove the basic identities they satisfy, as well as several related results, the most important of which are
a refined Weiss-type estimate and a consequent frequency pinching estimate, both originally introduced in \cite{DLMSV}. 
In Section \ref{GeometricSection} we recall the Jones $\beta_2$-number, its linear algebraic characterization, and a distortion bound in \cite{DLMSV},
which states that the Jones $\beta_2$-number can be controlled by the renormalized, averaged frequency pinching.

In our proof we also rely on the analysis of the zero set in \cite{AHL} by Hardt, Lin and the author.
We restate the most crucial lemmas from \cite{AHL} in Section \ref{TopologicalSection} for the reader's convenience. 
The key ideas in \cite{AHL} are the following: 
The zero set for homogeneous local minimizers in $\mathbf{R}^3$ was characterized and frequency bounds near the zero set were derived. 
(See Lemma \ref{3DHomogeneous} in Section \ref{TopologicalSection} for a precise statement.)
These were in turn used to analyze the zero set of general minimizers, as homogeneous blow-ups serve as good approximations of them at appropriate scales.
This strategy was combined with the following observation:
Since the vanishing order at any point with a cylindrical blow-up is $1/2\sqrt{\kappa}$ by the classification result \cite[Lemma 3.1]{AHL},
when the frequency at scale $2$ is sufficiently close to $1/2\sqrt{\kappa}$, the Hausdorff distance between the zero set and a line has to be small at scale $1$.
The combination of these two results enables one to locally verify at an appropriate scale the hypothesis of Reifenberg's Topological Disk Theorem, cf. \cite{Reifenberg60},
which yields a bi-H\"{o}lder homeomorphism. 

But even more crucially for the application in this article, the same combination also gives the frequency bound \eqref{SmallnessAssumption} locally and at an appropriate scale.
(While the bounds in \cite{AHL} are for the classical Almgren frequency function, since they depend on compactness arguments and the classification of homogeneous blow-ups only,
they remain valid for the smoothed frequency we will use in this article.)
Thanks to the bound \eqref{SmallnessAssumption}, instead of building intermediate and efficient covers as in \cite[Lemma 7.3, Proposition 7.2]{DLMSV},
we can directly build a cover that gives a local estimate on the Minkowski content of $u^{-1} \{ 0 \}$. 
Hence, even though the discrete Reifenberg theorem introduced in \cite{NV} still plays the central role in obtaining a Minkowski content estimate,
Section \ref{MeasureSection} is much simpler than its counterpart in \cite{DLMSV}.

Following the strategy of \cite{AHL} in Section \ref{FinalSection}, given an arbitrary point in the zero set, we are able to find a neighborhood of it, 
in which the zero set is close enough to that of a homogeneous map to start with. 
Then working at successively finer scales we verify the bound \eqref{SmallnessAssumption} locally on each H\"{o}lder arc emanating from this point. 
In particular, the length of the arc converges as a geometric sum, as we approach the focal point. 
While it would also be possible to follow the covering scheme in \cite{DLMSV} more closely, our proof is more attuned to the structure of the zero set, in the spirit of \cite{AHL},
and therefore, it is somewhat more geometric.

Finally, the topological persistence of zeros was observed in \cite[Lemma 4.2]{AHL}. In other words, the zero set has no holes. 
This observation played a crucial role in proving the main result of \cite{AHL}, which is contained in Theorem \ref{LocalFiniteLength}.
Unlike \cite{DLMSV}, we do not rely on the rectifiability result of Azzam and Tolsa in \cite{AT15}.
Since we already know that the zero set is a collection of continuous curves (modulo a discrete set of isolated zeros),
by a standard result in geometric measure theory, we obtain a Lipschitz parametrization for each of the finitely many curves in the zero set, cf. \cite[Theorem I.1.8]{DavidSemmes}.

\section{Variational formulas and related estimates} \label{VariationalSection}

We introduce the smoothed versions of the quantities that played a central role in the analysis in \cite{AHL}.
These smoothed quantities were defined for the first time in \cite{DLS}, and they were utilized in \cite{DLMSV},
which we follow closely in this section.

\begin{definition} \label{SmoothedFunctions}
Let $\phi$ be the following Lipschitz function:
$$
\phi(r) =
\begin{cases}
1, \quad & 0 \leq r \leq \frac{1}{2}, \\
2 - 2r, \quad & \frac{1}{2} \leq r \leq 1, \\
0, \quad & 1 \leq r.
\end{cases}
$$
We also let $\nu_x(y) = |y-x|^{-1} \left ( y - x \right )$, a unit vector field. 

We define the smoothed Dirichlet, height, frequency, remainder and frequency pinching functions respectively:
\begin{equation}
D_\phi (x,r) = \int_{\mathbf{R}^3} | \nabla u |^2 (y) \phi \left ( \frac{|x-y|}{r} \right ) \, \mathrm{d}y,
\label{SDirichlet}
\end{equation}
\begin{equation}
H_\phi (x,r) = - \int_{\mathbf{R}^3} |u|^2(y) |y-x|^{-1} \phi ' \left ( \frac{|x-y|}{r} \right ) \, \mathrm{d}y,
\label{SHeight}
\end{equation}
\begin{equation}
N_\phi (x,r) = \frac{r D_\phi (x,r)}{H_\phi (x,r)},
\label{SAlmgren}
\end{equation}
\begin{equation}
E_\phi(x,r) = - \int_{\mathbf{R}^3} \left | \partial_{\nu_x} u \right |^2(y) |y-x| \phi ' \left ( \frac{|x-y|}{r} \right ) \, \mathrm{d}y,
\end{equation}
\begin{equation}
W_s^r (x) = N_\phi (x,r) - N_\phi (x,s), \quad 0 < s \leq r.
\end{equation}
\end{definition}

In this section we consider $u \, : \, B_{64}(0) \to \mathbf{D}_\kappa$, a map minimizing the Dirichlet energy. 
Since our analysis is local in nature, we assume that $B_{64}(0) \subset \Omega$ and $N_{\phi}(0,64) \leq \overline{N}$. 
We will always apply the appropriately scaled versions of the variational identities and estimates in the following sections.

Note that if we replace $\phi$ with a sequence of smooth functions $\left \{ \phi_i \right \}$ approximating the characteristic function of $[0,1]$, 
then the corresponding quantities converge to the classical Dirichlet, height and Almgren frequency functions as introduced in \cite{Almgren83}.
Our motivations in working with smoothed quantities are twofold: 
Firstly, we would like to have differential identities analogous to those in \cite[Lemma 2.2]{AHL}, yet valid at every $r \in (0,\infty)$, not only at almost every $r$.
Secondly, we would like to consider directional derivatives of $D_\phi (x,r)$, $H_\phi (x,r)$ and $N_\phi (x,r)$ as well.
We state and prove these identities in the next lemma:

\begin{lemma} \label{VarIdentities}
Let $u \, : \, B_{64}(0) \to \mathbf{D}_\kappa$ be a map minimizing the Dirichlet energy.
The functions $D_\phi$, $H_\phi$ and $N_\phi$ are $C^1$ in spatial variables, as well as in $r \in (0, \infty)$.
Moreover, for any $r \in (0, + \infty )$, $x \in \mathbf{R}^3$, $v \in \mathbf{R}^3$, 
the following identities hold:
\begin{equation}
D_\phi(x,r) = - \frac{1}{r} \int_{\mathbf{R}^3} \partial_{\nu_x} u (y) \cdot u (y)  \phi' \left ( \frac{|x-y|}{r} \right ) \, \mathrm{d}y,
\label{IBP}
\end{equation}
\begin{equation}
\partial_r D_\phi (x,r) = \frac{1}{r} D_\phi (x,r) + \frac{2}{r^2} E_\phi (x,r),
\label{DirDerivScale}
\end{equation}
\begin{equation}
\partial_{v} D_\phi (x,r) = - \frac{2}{r} \int_{\mathbf{R}^3} \partial_{\nu_x} u (y) \cdot \partial_{v} u (y)  \phi' \left ( \frac{|x-y|}{r} \right ) \, \mathrm{d}y,
\label{DirDerivSpace}
\end{equation}
\begin{equation}
\partial_r H_\phi (x,r) = \frac{2}{r} H_\phi (x,r) + 2 D_\phi (x,r),  
\label{HeightDerivScale}
\end{equation}
\begin{equation}
\partial_{v} H_\phi (x,r) = - 2 \int_{\mathbf{R}^3} \partial_v u (y) \cdot u(y) | y - x |^{-1} \phi' \left ( \frac{|x-y|}{r} \right ) \, \mathrm{d}y.
\label{HeightDerivSpace}
\end{equation}
Moreover, $N_\phi (x,r)$ and $r^{-2} H_\phi (x,r)$ are nondecreasing functions of $r$, and they satisfy:
\begin{equation}
\partial_r N_\phi (x,r) = \frac{2}{r H_\phi (x,r)^2} \left ( H_\phi (x,r) E_\phi (x,r) - r^2 D_\phi (x,r)^2 \right ) \geq 0,
\label{FreqDerivScale}
\end{equation}
\begin{equation}
\partial_r \log \left ( r^{-2} H_\phi (x,r) \right ) = \frac{2 N_\phi (x,r)}{r}.
\label{LogHeight}
\end{equation}
\end{lemma}

\begin{proof}
We note that we can approximate $\phi$ by compactly supported, smooth functions $\phi_k$ in $W^{1,p}$ for every $p < \infty$, prove the above identities with $\phi_k$ in place of $\phi$,
and finally pass to the limit. Hence, it suffices to show the above identities for smooth $\phi$ with $\phi' \leq 0$. 
Since $u$ is a minimizer, we obtain the following first variation formulas by considering:
\begin{enumerate}[(a)]
 \item the target variation $u(y) \to \left ( 1 + t \phi(y) \right ) u(y)$, resulting in the Euler-Lagrange equation:
\begin{equation}
\int_{\mathbf{R}^3} \left [ \left \langle \nabla u \cdot u, \nabla \phi \right \rangle + \left | \nabla u \right |^2 \phi \right ] \, \mathrm{d}y = 0,
\label{TargetV}
\end{equation}
 \item the domain variation $u(y) \to u \left ( y + t \varphi(y)y \right )$, leading to the system of conservation laws:
\begin{equation}
\int_{\mathbf{R}^3} \left [ 2 \left \langle \nabla u \otimes \nabla u , \nabla \varphi \right \rangle - \left | \nabla u \right |^2 \mathrm{div} \varphi \right ] \, \mathrm{d}y = 0.
\label{DomainV}
\end{equation}
\end{enumerate}

Setting $\phi(y) = \phi \left ( \frac{|y-x|}{r} \right )$ in \eqref{TargetV} gives \eqref{IBP}.
Differentiating \eqref{SDirichlet} with respect to $r$ and setting $\varphi(y) = \phi \left ( \frac{|y-x|}{r} \right ) (y-x)$ in \eqref{DomainV} yields \eqref{DirDerivScale}.
Likewise, differentiating \eqref{SDirichlet} with respect to $x$, in the direction $v$, 
and setting $\varphi(y) = \phi \left ( \frac{|y-x|}{r} \right ) v$ in \eqref{DomainV} yields \eqref{DirDerivSpace}.

We can rewrite \eqref{SHeight} via a change of variables as:
\begin{equation}
H_\phi(x,r) = - \frac{1}{r^2} \int_{\mathbf{R}^3} \left | u(x+rz) \right |^2 |z|^{-1} \phi' \left ( |z| \right ) \, \mathrm{d}z. 
\label{SHeight2}
\end{equation}
Differentiating \eqref{SHeight2} with respect to $r$ and by changing back to the variable $y$, we obtain:
$$
\partial_r H_\phi(x,r) = \frac{2}{r} H_\phi (x,r) - 2 \int_{\mathbf{R}^3} \phi' \left ( \frac{|y-x|}{r} \right ) \partial_{\nu_x} u \cdot u \, \mathrm{d}y, 
$$
which combined with \eqref{IBP} implies \eqref{HeightDerivScale}.
Similary, we can rewrite \eqref{SHeight} via a change of variables as:
\begin{equation}
H_\phi(x,r) = - \frac{1}{r^2} \int_{\mathbf{R}^3} \left | u(x+z) \right |^2 |z|^{-1} \phi' \left ( \frac{|z|}{r} \right ) \, \mathrm{d}z. 
\label{SHeight3}
\end{equation}
Differentiating \eqref{SHeight3} with respect to $x$, in the direction $v$, and changing back to the variable $y$ yields \eqref{HeightDerivSpace}.

We observe that \eqref{FreqDerivScale} follows from \eqref{DirDerivScale} and \eqref{HeightDerivScale}, 
and the nonnegativity of $\partial_r N_\phi(x,r)$ can be seen using \eqref{IBP} and the Cauchy-Schwartz inequality.
Finally, \eqref{LogHeight} follows immediately from \eqref{HeightDerivScale}.
\end{proof}

Next we prove local bounds for $H_\phi$ and $N_\phi$ that are analogous to those proved in \cite[Lemma 2.3]{AHL} for the classical height and frequency functions:

\begin{lemma} \label{HFBounds}
There exists a constant $C(\phi) > 0$ such that for $u$ and $\phi$ as above, we have the following local bounds:
\begin{equation}
H_\phi (y, \rho ) \leq C H_\phi (x, 4 \rho),
\label{LocalHeight}
\end{equation}
whenever $y \in B_\rho(x)$ and $\mathrm{dist}(x, \partial \Omega) > 4 \rho$, and
\begin{equation}
N_\phi(y,r) \leq C \left ( N_\phi (x, 16r) + 1 \right ),
\label{LocalFreq} 
\end{equation}
whenever $y \in B_{r/4} (x)$ and $\mathrm{dist}(x, \partial \Omega) > 16r$.
\end{lemma}

\begin{proof}
Without loss of generality we can let $x = 0$ and $\rho = r =1$: 
Since $y \in B_1(x)$, $B_1(y) \subset B_2(0)$, and we have:
\begin{equation}
H_\phi(y,1) \leq C \int_{B_2(0)} |u|^2 \, \mathrm{d}z \leq \tilde{C} \int_{\partial B_s} |u|^2 \, \mathrm{d}S,
\label{SemiclassicalComparison}
\end{equation}
for all $s \in (2,4)$. Here the second inequality easily follows from the classical version of \eqref{LogHeight}, namely:
\begin{equation*}
\partial_r \log \left ( r^{-2} H(x,r) \right ) = \frac{2 N(x,r)}{r},
\end{equation*}
and the monotonicity of the classical frequency function $N(x,r) = \frac{r D(x,r)}{H(x,r)}$, where $H(x,r) = \int_{\partial B_r(x)} |u|^2 \, \mathrm{d}S$.
(The classical formulas can be obtained by replacing $\phi$ in Lemma \ref{VarIdentities}
with a sequence of smooth functions $\left \{ \phi_i \right \}$ approximating the characteristic function of $[0,1]$ and passing to the limit.)
Integrating the right-hand side of \eqref{SemiclassicalComparison} on $(0,4)$ with respect to the weight $-s^{-1} \phi'(s/4)$ yields: 
\begin{equation*}
H_\phi(y,1) \leq C H_\phi(0,4),
\end{equation*} 
which is \eqref{LocalHeight} after the normalizations $x=0$, $\rho =1$.

Applying \eqref{LocalHeight}, the integral version of \eqref{LogHeight}, and \eqref{FreqDerivScale} respectively, we have:
\begin{equation}
H_\phi(y,4) \leq C H_\phi (0,16)  \leq C e^{C N_\phi (0,16)} H_{\phi} (0,1/4) \leq C e^{C N_\phi (0,16)} H_{\phi} (y,1),
\label{NewComparison1}
\end{equation}
since $y \in B_{1/4}(0)$, whereas by \eqref{LogHeight} again:
\begin{equation}
H_\phi(y,1) = H_\phi(y,4) \exp \left [ - \int_1^4 \frac{2 N_\phi(y,t)}{t} \, \mathrm{d}t \right ].
\label{NewComparison2}
\end{equation}
As $H_\phi(y,4) > 0$, \eqref{NewComparison1} and \eqref{NewComparison2} together imply:
$$
\exp \left [ \int_1^4 \frac{2 N_\phi(y,t)}{t} \, \mathrm{d}t \right ] \leq C e^{C N_\phi (0,16) }.
$$
Taking the logarithm of both sides and using the monotonicity of $N_\phi(y,t)$ to estimate the integral on the left-hand side from below, we obtain:
$$
N_\phi (y,1) \leq C N_\phi(0,16) + \log C, 
$$
which yields \eqref{LocalFreq}.
\end{proof}

We recall the classical Weiss monotonicity formula introduced in \cite{Weiss} in the context of free boundary theory.
Defining:
$$
\mathcal{W}(x,r) = \frac{1}{r^{1+ 2 \alpha}} D \left ( x,r \right ) - \frac{\alpha}{r^{2 + 2 \alpha}} H \left ( x ,r \right ),
$$
where $N \left ( x ; 0^+ \ \right ) = \alpha$, for almost every $r > 0$:
\begin{equation}
\begin{aligned}
\frac{d}{dr} \mathcal{W}\left ( x_0 , r \right ) = 
\frac{2}{r^{3 + 2 \alpha}} \int_{\partial B_r\left ( x_0 \right )} \left | \left ( x - x_0 \right ) \cdot \nabla u - \alpha u \right |^2 \, \mathrm{d}S.
\label{ClassicalWeissDiff}
\end{aligned}
\end{equation}
The proof follows immediately from the classical versions of \eqref{DirDerivScale} and \eqref{HeightDerivScale}. 
Integrating \eqref{ClassicalWeissDiff} on $[r,R]$, we obtain the following classical estimate:
\begin{equation}
\begin{aligned}
\int_r^R \frac{2}{r^{3 + 2 \alpha }} \int_{\partial B_r\left ( x_0 \right )} \left | \left ( x - x_0 \right ) \cdot \nabla u - \alpha u \right |^2 \, 
\mathrm{d}S \, \mathrm{d}r
= \mathcal{W}\left ( x_0 , R \right ) - \mathcal{W}\left ( x_0 , r \right ).
\label{ClassicalWeissIntegral}
\end{aligned}
\end{equation}
We remark that the left-hand side vanishes, if and only if $u$ is homogeneous of order $\alpha$ on the annulus $B_R \left ( x_0 \right ) \backslash B_r \left ( x_0 \right )$.

The next lemma is a refined version of \eqref{ClassicalWeissIntegral}:

\begin{lemma} \label{QuantWeiss}
There exists a $C = C \left ( \overline{N} \right ) > 0$ such that for $u$ and $\phi$ as above, for every $x \in B_{1/8} (0)$:
\begin{equation}
\int_{B_2 ( x ) \backslash B_{1/4} ( x  )} 
\left | \left ( y - x\right ) \cdot \nabla u - N_\phi \left ( x , \left | y-x \right | \right ) u \right |^2 \, \mathrm{d}y \leq C W_{1/8}^4 \left ( x \right ).
\label{QuantWeissEstimate}
\end{equation}
\end{lemma}

\begin{proof}
Without loss of generality we can assume that $H_\phi(0,1) = 1$.
Firstly, we claim:
\begin{equation}
\begin{aligned}
\int_{B_2(x) \backslash B_{1/4}(x)} \left | (y-x) \cdot \nabla u - N_{\phi} \left ( x, |y-x| \right ) u \right |^2 \, \mathrm{d}y \\
\leq 4 \int_{1/4}^4 \int_{B_\tau(x) \backslash B_{\tau/2}(x)} \left | (y-x) \cdot \nabla u - N_{\phi} \left ( x, |y-x| \right ) u   \right |^2 \, \mathrm{d}y \, \mathrm{d} \tau.
\label{FubiniReduction}
\end{aligned}
\end{equation}
Setting:
$$
F(s) = \int_{\partial B_s(x)} \left | (y-x) \cdot \nabla u - N_{\phi} \left ( x, |y-x| \right ) u \right |^2 \, \mathrm{d}S,
$$
by Fubini's Theorem and the nonnegativity of $F(s)$, we have:
\begin{equation*}
\begin{aligned}
\int_{\frac{1}{4}}^4 \int_{\frac{\tau}{2}}^\tau F(s) \, \mathrm{d}s \, \mathrm{d}\tau & = 
\int_{\frac{1}{8}}^{\frac{1}{4}} \int_{1/4}^{2s} F(s) \, \mathrm{d}\tau \, \mathrm{d}s + \int_{\frac{1}{4}}^2 \int_s^{2s} F(s) \, \mathrm{d}\tau \, \mathrm{d}s + 
\int_2^4 \int_s^{4} F(s) \, \mathrm{d}\tau \, \mathrm{d}s \\
& \geq \int_{\frac{1}{4}}^2 \int_s^{2s} F(s) \, \mathrm{d}\tau \, \mathrm{d}s \geq \int_{\frac{1}{4}}^2 s F(s) \, \mathrm{d}s \geq \frac{1}{4} \int_{\frac{1}{4}}^2 F(s) \, \mathrm{d}s.
\end{aligned}
\end{equation*}
Therefore, \eqref{FubiniReduction} follows from multiplying both sides with $4$ and invoking Fubini's Theorem once again.

From \eqref{FubiniReduction} we get:
\begin{equation}
\begin{aligned}
\int_{B_2(x) \backslash B_{1/4}(x)} \left | (y-x) \cdot \nabla u - N_{\phi} \left ( x, |y-x| \right ) u \right |^2 \, \mathrm{d}y \\
\leq 8 \int_{1/4}^4 \int_{B_\tau(x) \backslash B_{\tau/2}(x)} \left | (y-x) \cdot \nabla u - N_{\phi} \left ( x, \tau \right ) u   \right |^2 \, \mathrm{d}y \, \mathrm{d} \tau + \\
     8 \int_{1/4}^4 \int_{B_\tau(x) \backslash B_{\tau/2}(x)} 
     \left | N_{\phi} \left ( x, \tau \right ) - N_{\phi} \left ( x, |y-x| \right ) \right |^2 \left | u   \right |^2 \, \mathrm{d}y \, \mathrm{d} \tau.
\end{aligned}
\label{TwoTerms}
\end{equation}

We observe that for $\eta_x(y) = y-x = |y-x| \nu_x(y)$, the first term on the right-hand side of \eqref{TwoTerms} can be estimated by a constant multiple of the following:
\begin{equation}
\begin{aligned}
\int_{\frac{1}{4}}^4 \int_{\mathbf{R}^3} \frac{-1}{|y-x|} \phi' \left ( \frac{|y-x|}{\tau} \right )
\left [ \left | \partial_{\eta_x} u \right |^2 - 2 N_\phi (x, \tau) \partial_{\eta_x} u \cdot u + N_\phi(x,\tau)^2 |u|^2 \right ] \, \mathrm{d}y \, \mathrm{d} \tau \\
= \int_{1/4}^4 \left [ E_\phi(x,\tau) - 2 \tau N_\phi(x,\tau) D_\phi(x,\tau) + N_\phi(x,\tau)^2 H_\phi(x,\tau) \right ] \, \mathrm{d}\tau \\
= \int_{1/4}^4 \left [ E_\phi(x,\tau) - \tau N_\phi(x,\tau) D_\phi(x,\tau) \right ] \, \mathrm{d}\tau .
\end{aligned}
\label{FirstTermRe}
\end{equation}
Since $H_\phi(0,1) = 1$, by \eqref{LogHeight} and \eqref{LocalHeight}, we can bound $H_\phi(x,\tau)$ uniformly 
from above by a constant depending on $\overline{N}$, whenever $\tau \in [1/4,4]$.
As a result, we can estimate the right-hand side of \eqref{FirstTermRe} by a constant multiple of the following:
\begin{equation}
\begin{aligned}
\int_{1/4}^4 \frac{2}{\tau H_\phi (x, \tau)} \left [ E_\phi(x,\tau) - \tau N_\phi(x,\tau) D_\phi(x,\tau) \right ] \, \mathrm{d}\tau & =
\int_{1/4}^4 \partial_\tau N_\phi(x,\tau) \, \mathrm{d} \tau \\ & = W_{1/4}^4(x) \leq W_{1/8}^4(x),
\end{aligned}
\label{FirstTermEst}
\end{equation}
where the last two inequalities are due to \eqref{FreqDerivScale} and the monotonicity of $N_\phi(x,r)$ respectively. 
Hence, from \eqref{FirstTermRe} and \eqref{FirstTermEst},
we conclude that the first term on the right-hand side of \eqref{TwoTerms} is bounded from above by $C W_{1/4}^4(x) $.

Finally, we can estimate the second term on the right-hand side of \eqref{TwoTerms} by a constant multiple of:
\begin{equation}
\left [ N_{\phi} \left ( x, 4 \right ) - N_{\phi} \left ( x, 1/8 \right )  \right ]^2 H_\phi(0,1) \leq C \left ( \overline{N} \right ) W_{1/8}^4(x),
\end{equation}
by the monotonicity and positivity of $N_\phi(x,r)$, \eqref{LogHeight}, \eqref{LocalHeight} and the normalization $H_\phi(0,1) = 1$.
Consequently, the proof of \eqref{QuantWeissEstimate} is complete.
\end{proof}

An important consequence of the quantitative Weiss-type estimate \eqref{QuantWeissEstimate} is the following frequency pinching estimate introduced in \cite{DLMSV}.

\begin{lemma} \label{Pinching}
There exists a $C = C \left ( \overline{N} \right ) > 0$ such that for $u$ and $\phi$ as above, $x_1$, $x_2 \in B_{1/8} (0)$ and $ \left | x_1 - x_2 \right | \leq r/4$, 
whenever $y$ and $z$ lie on the segment joining $x_1$ and $x_2$, we have:
\begin{equation}
\left | N_\phi ( y, r ) - N_\phi (z, r) \right | \leq 
C \left [ \left ( W_{r/8}^{4r} \left ( x_1 \right ) \right )^{1/2} + \left ( W_{r/8}^{4r} \left ( x_2 \right ) \right )^{1/2} \right ] |y-z|/r.
\label{PinchingEstimate}
\end{equation}
\end{lemma}

\begin{proof}
We can assume without loss of generality that $r = 1$ and $H_\phi(0,1) = 1$.
We would like to bound $\left | \partial_v N_\phi (x,1) \right |$ uniformly in $x$, where $v = x_2 - x_1$ and $x$ lies on the segment joining $x_1$ and $x_2$. 

By \eqref{DirDerivSpace} and \eqref{HeightDerivSpace}, we have:
\begin{equation}
\partial_v N_\phi (x,1) = \frac{2}{H_\phi (x,1)} \int_{\mathbf{R}^3} \left [ \partial_v u \cdot  \partial_{\eta_x} u - N_\phi (x,1) \partial_v u \cdot u \right ] \, \mathrm{d}\mu_x,
\label{PinchingDeriv1}
\end{equation}
where $\eta_x(y) = y - x = |y-x| \nu_x (y)$ and $\mu_x = - |y-x|^{-1} \phi' \left ( |y-x| \right ) \, \mathrm{d}y$.
We can write:
$$
\partial_v u(z) = \partial_{\eta_{x_1}} u(z) - \partial_{\eta_{x_2}} u(z) = \mathcal{E}_1(z) - \mathcal{E}_2(z) + \Xi (z) u(z),
$$
where:
\begin{equation*}
\begin{aligned}
\mathcal{E}_1(z) = \partial_{\eta_{x_1}} u(z) - N_\phi \left ( x_1 , \left | z - x_1 \right | \right ) u(z), \\
\mathcal{E}_2(z) = \partial_{\eta_{x_2}} u(z) - N_\phi \left ( x_2 , \left | z - x_2 \right | \right ) u(z), \\
\Xi (z) =  N_\phi \left ( x_1, \left | z - x_1 \right | \right ) - N_\phi \left ( x_2, \left | z - x_2 \right | \right ).
\end{aligned}
\end{equation*}
Substituting the expression for $\partial_v u$ in \eqref{PinchingDeriv1}, we obtain:
$$
\partial_v N_\phi (x,1) = I + II + III,
$$
where:
\begin{equation*}
\begin{aligned}
I = \frac{2}{H_\phi (x,1)} \int_{\mathbf{R}^3} \left ( \mathcal{E}_1 - \mathcal{E}_2 \right ) \cdot \partial_{\eta_x} u \, \mathrm{d} \mu_x, \\
II = \frac{2 D_\phi (x,1)}{H_\phi (x,1)^2} \int_{\mathbf{R}^3} \left ( \mathcal{E}_2 - \mathcal{E}_1 \right ) \cdot u \, \mathrm{d} \mu_x, \\
III = \frac{2}{H_\phi (x,1)} \int_{\mathbf{R}^3} \left [ \partial_{\eta_x} u \cdot u - N_\phi(x,1) |u|^2 \right ] \Xi(y) \, \mathrm{d} \mu_x. 
\end{aligned}
\end{equation*}

By the Cauchy-Schwartz inequality:
\begin{equation}
I^2 \leq \frac{4}{H_\phi (x,1)^2} 
\int_{\mathbf{R}^3} \left | \mathcal{E}_1 - \mathcal{E}_2 \right |^2 \, \mathrm{d}\mu_x \int_{\mathbf{R}^3} \left | \nabla u \right |^2 \, \mathrm{d}\mu_x.
\label{PinchingTerm1}
\end{equation}
Moreover, by \eqref{LocalFreq}, \eqref{LocalHeight}, \eqref{LogHeight} and $H_\phi(0,1) = 1$, 
we can bound $H_\phi(x,\tau)$ from below and $N_\phi(x,\tau)$ from above by uniform constants depending on $\overline{N}$, whenever $\tau \in [1/4,4]$.
As a result, \eqref{PinchingTerm1} implies: 
\begin{equation}
|I| \leq C \left ( \int_{\mathbf{R}^3} \left [ \left | \mathcal{E}_1 \right |^2 + \left | \mathcal{E}_2 \right |^2 \right ] \, \mathrm{d} \mu_x \right )^{1/2}.
\label{PinchingTerm1Int}
\end{equation}
Using the same ingredients as above, we also obtain the following estimate:
\begin{equation}
|II| \leq C \left ( \int_{\mathbf{R}^3} \left [ \left | \mathcal{E}_1 \right |^2 + \left | \mathcal{E}_2 \right |^2 \right ] \, \mathrm{d} \mu_x \right )^{1/2}.
\label{PinchingTerm2Int}
\end{equation}

Next we claim:
\begin{equation}
\int_{\mathbf{R}^3} \left | \mathcal{E}_i \right |^2 \, \mathrm{d}\mu_x \leq C W_{1/8}^4 \left (x_i \right ), \quad i=1,2.
\label{PinchingTerms}
\end{equation}
We can express the left-hand side of \eqref{PinchingTerms} as follows:
\begin{equation*}
\int_{\mathbf{R}^3} \left | \mathcal{E}_i \right |^2 \, \mathrm{d}\mu_x = 
\int_{\mathbf{R}^3} \frac{-1}{|y-x|} \phi' \left ( |y-x| \right ) \left | \mathcal{E}_i \right |^2 \, \mathrm{d}y =
\int_{\mathbf{R}^3} \left | \mathcal{E}_i \right |^2 m(y) \, \mathrm{d}y.
\end{equation*}
We observe that since $\phi'(s) = -2$ on $s \in [1/2,1]$ and $0$ elsewhere, $0 \leq \frac{-1}{s}\phi'(s) \leq 4$.
Moreover, since $x$ lies on the segment joining $x_1$ and $x_2$, 
$$ | y - x | \leq \left | y - x_i \right | + \left | x_i - x \right | \leq 1/4 + 1/4 = 1/2,$$ 
when $\left | y - x_i \right | \leq 1/4 $ for $i=1$ or $2$.
Likewise,
$$ | y - x | \geq \left | y - x_i \right | - \left | x_i - x \right | \geq 2 - 1/4 \geq 1, $$
when $ \left | y - x_i \right | \geq 2$ for $i=1$ or $2$.
Hence, $m(y)$ vanishes on the complement of $B_2 \left ( x_i \right ) \backslash B_{1/4} \left ( x_i \right )$.
Therefore, we get:
\begin{equation*}
\int_{\mathbf{R}^3} \left | \mathcal{E}_i \right |^2 \, \mathrm{d}\mu_x \leq 
4 \int_{B_2 \left ( x_i \right ) \backslash B_{1/4} \left ( x_i \right )} \left | \left ( y - x_i \right ) \cdot \nabla u - N_\phi \left ( x_i, \left | y - x_i \right | \right ) u \right |^2
\, \mathrm{d}y,
\end{equation*}
and \eqref{PinchingTerms} follows immediately from \eqref{QuantWeissEstimate}.

We note that \eqref{PinchingTerm1Int}, \eqref{PinchingTerm2Int} and \eqref{PinchingTerms} imply:
\begin{equation*}
|I| + |II| \leq C \left ( W_{1/8}^4 \left ( x_1 \right ) +  W_{1/8}^4 \left ( x_2 \right ) \right )^{1/2} 
\leq C \left [ \left ( W_{1/8}^4 \left ( x_1 \right ) \right )^{1/2} +  \left ( W_{1/8}^4 \left ( x_2 \right ) \right )^{1/2} \right ],
\end{equation*}
and there remains to estimate $III$.

We can write $\Xi = A + B + C$, where:
\begin{equation*}
\begin{aligned}
A = N_\phi \left ( x_1 , 1 \right ) - N_\phi \left ( x_2, 1 \right ), \\
B(z) = N_\phi \left ( x_1 , \left | z - x_1 \right | \right ) - N_\phi \left ( x_1, 1 \right ), \\
C(z) = N_\phi \left ( x_2, 1 \right ) - N_\phi \left ( x_2 , \left | z - x_2 \right | \right ).
\end{aligned}
\end{equation*}
From \eqref{IBP} we get:
\begin{equation}
\begin{aligned}
\int_{\mathbf{R}^3} \left [ \partial_{\eta_x} u \cdot u - N_\phi(x,1) |u|^2 \right ] A \, \mathrm{d} \mu_x & =
A \left [ \int_{\mathbf{R}^3} \partial_{\eta_x} u \cdot u \, \mathrm{d} \mu_x -  N_\phi(x,1) \int_{\mathbf{R}^3} |u|^2 \, \mathrm{d} \mu_x \right ] \\
& = A \left [ - \int \phi' \left ( |y-x| \right ) \partial_{\nu_x} u \cdot u \, \mathrm{d}y - D_\phi(x,1) \right ] \\
& = 0.
\end{aligned}
\label{NullTerm}
\end{equation}
Secondly, since $\mu_x$ is supported in $B_1(x) \backslash B_{1/2}(x)$, and $x$ lies on the segment joining $x_1$ and $x_2$, while $\left | x_1 - x_2 \right | \leq 1/4$,
we observe: 
$$ 
1/8 < | z - x | - \left | x_i - x \right | \leq \left | z - x_i \right | \leq | z - x | + \left | x - x_i \right | < 4, 
$$
for $i=1,2$, which implies for every $z \in B_1(x) \backslash B_{1/2}(x)$ and every $x$ on the segment joining $x_1$ and $x_2$:
\begin{equation}
|B(z)| + |C(z)| \leq W_{1/8}^4 \left ( x_1 \right ) +  W_{1/8}^4 \left ( x_2 \right ).
\label{NonNullTerms}
\end{equation}
Hence, using \eqref{NullTerm}, \eqref{NonNullTerms}, \eqref{LocalFreq} and \eqref{LogHeight}, we estimate:
\begin{equation}
\begin{aligned}
|III| & \leq  2 \frac{ \left [ W_{1/8}^4 \left ( x_1 \right ) +  W_{1/8}^4 \left ( x_2 \right ) \right ] }{H_\phi (x,1)}
\left [ N_\phi(x,1) \int_{\mathbf{R}^3} |u|^2 \, \mathrm{d}\mu_x  + \int_{\mathbf{R}^3} |u||\nabla u| \, \mathrm{d}\mu_x \right ] \\
& \leq 2 \frac{ \left [ W_{1/8}^4 \left ( x_1 \right ) +  W_{1/8}^4 \left ( x_2 \right ) \right ] }{H_\phi (x,1)}
\left [ \left ( 1 + N_\phi(x,1) \right ) H_\phi(x,1) + C D_\phi (x,2) \right ] \\
& \leq C \left ( 1 + N_\phi(x,1) + H_\phi (x,1)^{-1} N_\phi(x,2) H_\phi(x,2)  \right )  \left [ W_{1/8}^4 \left ( x_1 \right ) +  W_{1/8}^4 \left ( x_2 \right ) \right ] \\
& \leq C \left ( 1 + N_\phi(x,1) + C \left ( \overline{N} \right ) N_\phi(x,2) \right ) \left [ W_{1/8}^4 \left ( x_1 \right ) +  W_{1/8}^4 \left ( x_2 \right ) \right ] \\
& \leq C \left ( \overline{N} \right ) \left [ W_{1/8}^4 \left ( x_1 \right ) +  W_{1/8}^4 \left ( x_2 \right ) \right ].
\end{aligned}
\end{equation}
Thus, using \eqref{LocalFreq} once again, we conclude that for every $x$ on the segment joining $x_1$ and $x_2$:
\begin{equation*}
\left | \partial_v N_\phi (x,1) \right | \leq |I| + |II| + |III| \leq C \left ( \overline{N} \right ) 
\left [ \left ( W_{1/8}^4 \left ( x_1 \right ) \right )^{1/2} + \left ( W_{1/8}^4 \left ( x_2 \right ) \right )^{1/2} \right ],
\end{equation*}
which immediately implies \eqref{PinchingEstimate}.

\end{proof}

We need the following unique continuation lemma as a technical ingredient in proving further frequency pinching estimates:

\begin{lemma} \label{UniqueContinuation}
If $u$ and $v$ are energy-minimizing harmonic maps into $\mathbf{D}_\kappa$ in a connected domain $\Omega$ and $u = v$ in an open set $U \subset \Omega$,
then $u = v$ in $K$ for any $K \subset \subset \Omega$.
\end{lemma}

\begin{proof}
Both $u^{-1} \{ 0 \}$ and $v^{-1} \{ 0 \}$ have Hausdorff dimension less than or equal to $1$ in $\Omega$ (cf. \cite[Lemma 2.11]{AHL}.)
Therefore, $\Omega' = \Omega \backslash \left ( u^{-1} \{0\} \cup u^{-1} \{0\} \right )$ is open and connected.
Note that in any open subset of $\Omega'$, $u$ and $v$ are smooth solutions of the harmonic map equation:
\begin{equation}
\Delta w + A_w \left ( \nabla w , \nabla w \right ) = 0,
\label{HarmonicMap}
\end{equation}
where $A_y$ denotes the second fundamental form of $\mathbf{D}_\kappa$ at $y$, cf. \cite[Appendix 2.12.3]{SimonBook} for details.
Then by the classical unique continuation theorem for elliptic systems in \cite{Aronszajn}, $u = v$ in $\Omega'$. 
Moreover, $u$ and $v$ are locally H\"older continuous in $\Omega$ (cf. \cite[Lemma 2.8]{AHL}.) 
Hence, by continuity $u = v$ in any $K \subset \subset \Omega$. 
\end{proof}

We end this section with another useful technical result from \cite{DLMSV}.
This lemma can be interpreted as a refined version of \cite[Lemma 4.2]{AHL}.

\begin{lemma} \label{FrequencyControl}
For $u$ and $\phi$ as above and $\rho, \tilde{\rho}, \overline{\rho}, \hat{\rho} \in (0,1)$ given, for all $\delta > 0$, there exists an
$ \epsilon = \epsilon \left ( \overline{N}, \rho, \tilde{\rho}, \overline{\rho}, \hat{\rho}, \delta \right ) > 0$ such that, if $x_1, x_2 \in B_1(0)$ satisfy
$\mathrm{dist} \left ( x_1, x_2 \right ) \geq \rho$, and
\begin{equation} 
W^2_{\tilde{\rho}} \left ( x_i \right ) = N_{\phi} \left ( x_i , 2 \right ) - N_{\phi} \left ( x_i , \tilde{\rho} \right ) < \epsilon, \quad i = 1,2,
\label{SmallDrop}
\end{equation}
then:
\begin{equation}
u^{-1} \{ 0 \} \cap \left ( B_1(0) \backslash B_{\hat{\rho}} (V) \right ) = \emptyset,
\label{HalfReifenberg} 
\end{equation}
and for all $y, y' \in B_1(0) \cap V$ and for all $r, r' \in \left [ \overline{\rho}, 1 \right ]$:
\begin{equation}
\left | N_{\phi} (y,r) - N_{\phi} \left ( y', r' \right ) \right | < \delta, 
\label{DropControl}
\end{equation}
where $V = x_1 \, + \, \mathrm{span} \left \{ x_2 - x_1 \right \}$.
\end{lemma}

\begin{proof}
The basic idea is to assume by contradiction that the lemma does not hold, to pass to a limit map by a compactness argument, and to show that either \eqref{HalfReifenberg} or \eqref{DropControl}
failing for the limit map gives a contradiction. We will prove the existence of an $\epsilon$ ensuring \eqref{HalfReifenberg} first, then prove the existence of an $\epsilon$ ensuring
\eqref{DropControl}. Clearly, taking the minimum of the two will imply the claim.

Firstly, suppose there exist positive $\rho$, $\tilde{\rho}$ and  $\hat{\rho}$ for which \eqref{HalfReifenberg} fails to hold for every $\epsilon > 0$.
Then we have a sequence of maps $u_k$, which are energy minimizing in $B_{64}(0)$, satisfying $N_\phi(0,64) \leq \overline{N}$, $H_\phi(0,64) = 1$, 
and for a sequence of pairs $\left \{ x_1^k, x_2^k \right \} \subset B_1(0)$ such that 
$\mathrm{dist} \left ( x_1, x_2 \right ) \geq \rho$, \eqref{SmallDrop} holds with $\epsilon_k \to 0$ and fixed $\tilde{\rho}$,
while there exist $y^k \in u_k^{-1} \{ 0 \} \cap \left ( B_1(0) \cap B_{\hat{\rho}} \left ( V^k \right ) \right )$ with fixed $\hat{\rho}$ 
and $V^k = x_1^k + \mathrm{span} \left \{ x_2^k - x_1^k \right \}$.

Note that once we fix an arbitrary $K \subset \subset B_{64}(0)$, the bounds $N_{\phi, k}(0,64) \leq \overline{N}$ and $H_{\phi, k}(0,64) = 1$ give a uniform $H^1(K)$ bound.
We apply the compactness lemma \cite[Lemma 2.7]{AHL} to obtain convergence (of a subsequence) to a limit map $u$, which is an energy minimizer in $K \subset \subset B_{64}(0)$, 
and where the convergence is strong in $H^1 \left ( K \right )$ and uniform in $K$. 
Furthermore, using the classical version of \eqref{HeightDerivScale}, (cf. \cite[Equation (2.4)]{AHL}), we can also show that the convergence is strong in $L^2 \left ( B_{64}(0) \right )$, 
which together with $H_{\phi, k}(0,64) = 1$ and \eqref{HeightDerivScale} imply that the limit $u$ is nontrivial, $H_\phi (x, s ) > 0$, and $N_{\phi, k}(z,s)$ converge to the smoothed
frequency for the limit, $N_{\phi}(z,s)$, whenever $s < 64 - |z|$.

Thus, we obtain $x_1, x_2, y \in \overline{B_1(0)}$ such that $\mathrm{dist} \left ( x_1, x_2 \right ) \geq \rho$, $u(y) = 0$, and $u$ satisfies:
$$
N_{\phi} \left ( x_i , 2 \right ) = N_{\phi} \left ( x_i , \tilde{\rho} \right ), \quad i = 1,2.
$$
Hence, by \eqref{FreqDerivScale} and the Cauchy-Schwartz inequality, 
$u$ is homogeneous of degree $\alpha_i$ with respect to $x_i$ in $B_2 \left ( x_i \right ) \backslash B_{\tilde{\rho}} \left ( x_i \right )$ for $i=1,2$. 

We can consider the homogeneous extensions of degree $\alpha_1$, $\alpha_2$ to $\mathbf{R}^3$ with respect to $x_1$ and $x_2$ respectively, and by continuity and Lemma \ref{UniqueContinuation}, 
each must agree with $u$ on $B_{64}(0)$. Hence, we get an extension of $u$ as a locally minimizing map defined in $\mathbf{R}^3$. 
Note that $u \left ( x_1 \right ) = u \left ( x_2 \right )  = u (y) = 0$.
Moreover, one can show by elementary means that $\alpha_1 = \alpha_2$ must hold and that $u$ is independent of the direction $v = x_2 - x_1$. 
However, this contradicts the fact that $u^{-1} \{ 0 \}$ has Hausdorff dimension $1$ at most.

By an analogous argument we can show that by shrinking $\epsilon$ to be small enough with respect to $\overline{\rho}$, $\tilde{\rho}$, $\rho$ and $\delta$, 
we can ensure that \eqref{DropControl} holds as well. Suppose that the claim fails to hold for some $\delta$, $\rho$, $\tilde{\rho}$, $\overline{\rho}$ positive.
This time alongside sequences $u_k$, $\left \{ x_1^k, x_2^k \right \}$, $\epsilon_k$ as above, we have pairs
$\left \{ y_1^k, y_2^k \right \}$ and radii $r_i$, $r'_i$ for which \eqref{DropControl} fails. 
Passing to a limit by arguing as above above we get $y_1, y_2 \in \overline{B_1(0)} \cap V$, where $V = x_1 \, + \, \mathrm{span} \left \{ x_2 - x_1 \right \}$
and $r, r' \in \left [ \overline{\rho}, 1 \right ]$ such that:
\begin{equation}
\left | N_\phi \left ( y_1, r \right ) -  N_\phi \left ( y_2, r' \right ) \right | \geq \delta,
\label{Conclusion1}
\end{equation}
whereas:
\begin{equation}
N_{\phi} \left ( x_i , 2 \right ) = N_{\phi} \left ( x_i , \tilde{\rho} \right ), \quad i = 1,2.
\label{Conclusion2}
\end{equation}
However, \eqref{Conclusion2} allows us to consider homogeneous extensions of $u$ as above. 
Once again using the continuity of the limit map $u$ and Lemma \ref{UniqueContinuation}, we can show that the limit map $u$ is independent of the direction $v = x_2 - x_1$,
and homogeneous of degree $\alpha$. 
In particular, we arrive at the conclusion that for the limit map $u$, $N_\phi (x,r) = \alpha$, for every $r > 0$ and any $x \in V$ and, contradicting \eqref{Conclusion1}. 
\end{proof}

\section{Distortion bound} \label{GeometricSection}

As in \cite{NV} and \cite{DLMSV}, an important quantity in our analysis is the Jones $\beta_2$-number, 
which was originally defined in \cite{Jones} in the context of the analyst's traveling salesman problem in $\mathbf{R}^2$. 

\begin{definition}
For $\mu$ a nonnegative Radon measure in $\mathbf{R}^3$, $x \in \mathbf{R}^3$ and $r > 0$, we define the Jones $\beta_2$-number:
\begin{equation}
D_\mu(x,r) = \inf_{L} \frac{1}{r^3} \int_{B_r(x)} \mathrm{dist} (y,L)^2 \, \mathrm{d} \mu(y),
\label{JonesBeta}
\end{equation}
where the infimum is with respect to all lines in $\mathbf{R}^3$ and $\mathrm{dist} (y,L) = \inf_{z \in L} |z-y|$. 
\end{definition}

We need to characterize the Jones $\beta_2$-number and the optimal lines $L$ in \eqref{JonesBeta} in a linear algebraic fashion, which yields to estimates in terms of frequency pinching.
For $x_0 \in \mathbf{R}^3$ and $ r_0 > 0$ such that $\mu \left ( B_{r_0} \left ( x_0 \right ) \right ) > 0$, we define the barycenter of $\mu$ in $B_{r_0} \left ( x_0 \right )$:
$$
\overline{x}_{x_0,r_0} = \frac{1}{\mu \left ( B_{r_0} \left ( x_0 \right ) \right )} \int_{B_{r_0} \left ( x_0 \right )} x\, \mathrm{d} \mu(x).
$$
We also define the bilinear form $B \, : \mathbf{R}^3 \times \mathbf{R}^3 \to \mathbf{R}$ as:
$$
B(v,w) = \int_{B_{r_0} \left ( x_0 \right )} \left ( \left ( x - \overline{x}_{x_0,r_0} \right ) \cdot v \right ) \left ( \left ( x - \overline{x}_{x_0,r_0} \right ) \cdot w \right ) 
\, \mathrm{d}\mu(x), \quad v,w \in \mathbf{R}^3.
$$
Note that $B$ is symmetric and positive semi-definite. Hence, there exists an orthonormal basis $\left \{ v_1, v_2, v_3 \right \}$ of $\mathbf{R}^3$ such that
$ B \left ( v_i , v_j \right ) = \lambda_i \delta_{ij}$. In particular:
\begin{equation}
\int_{B_{r_0} \left ( x_0 \right )} \left ( \left ( x - \overline{x}_{x_0,r_0} \right ) \cdot v_i \right ) x \, \mathrm{d}\mu(x) = \lambda_i v_i, \quad i=1,2,3.
\label{Eigenvectors}
\end{equation}
Consequently, for $x_0 \in \mathbf{R}^3$ and $r_0 > 0$ such that $\mu \left ( B_{r_0} \left ( x_0 \right ) \right ) > 0$, we observe that:
\begin{equation}
D_{\mu} \left ( x_0, r_0 \right ) = \frac{\lambda_2 + \lambda_3}{r_0^3} \leq 2 \frac{\lambda_2}{r_0^3},
\label{Eigenvalues}
\end{equation}
where $\lambda_1 \geq \lambda_2 \geq \lambda_3$ are the eigenvalues corresponding to the eigenvectors in \eqref{Eigenvectors}, 
and the infimum in \eqref{JonesBeta} is achieved by the line $L = \overline{x}_{x_0,r_0} + \mathrm{span} \left \{ v_1 \right \}$.

The following distortion bound was proved in \cite[Proposition 5.3]{DLMSV}, and it will be an important ingredient in estimating the Minkowski content of $u^{-1} \{ 0 \}$ locally 
in Section \ref{MeasureSection}.

\begin{lemma} \label{BetaEstimate}
There exists a $C = C \left ( \overline{N} \right ) > 0$ such that if $\mu$ is a finite, nonnegative Radon measure with $\mathrm{spt}(\mu) \subset \mathcal{Z}$,
then for every $x_0 \in B_{1/8}(0)$ and $r \in (0,1]$:
\begin{equation}
D_\mu \left ( x_0 , r/8 \right ) \leq \frac{C}{r} \int_{B_{r/8} \left ( x_0 \right ) } W^{4r}_{r/8} (x) \, \mathrm{d} \mu(x).
\label{BetaBound}
\end{equation}
\end{lemma}

\begin{proof}
We give the proof in \cite{DLMSV} for the convenience of the reader. 
A simple scaling argument shows that it suffices to prove \eqref{BetaBound} for $r=1$ and $H_\phi (0,1) = 1$.
Unless \eqref{BetaBound} holds trivially, we can assume $\mu \left ( B_{1/8}(0) \right ) > 0$, and consequently $u^{-1} \{0\} \cap B_{1/8}(0) \neq \emptyset$. 
Let $\overline{x}$ be the barycenter of $\mu$ in $B_{1/8} \left ( x_0 \right )$, and $ \{ v_1, v_2, v_3 \}$ an eigenbasis of $\mathbf{R}^3$, satisfying \eqref{Eigenvectors}
with $\lambda_1 \geq \lambda_2 \geq \lambda_3 \geq 0$. 
Then for every $j=1,2,3$, every $y \in B_{3/2} \left ( x_0 \right ) \backslash B_{1/2} \left ( x_0 \right )$ and any constant $\alpha$, we have:
\begin{equation}
- \lambda_j v_j \cdot \nabla u(y) = \int_{B_{1/8} \left ( x_0 \right )} 
\left ( \left ( x- \overline{x} \right ) \cdot v_j \right ) \left ( (y-x) \cdot \nabla u - \alpha u \right ) \, \mathrm{d} \mu(x).
\label{AlphaIdentity}
\end{equation}
Squaring both sides of \eqref{AlphaIdentity}, applying the Cauchy-Schwartz inequality to the right-hand side, using the identity $B \left ( v_j, v_j \right ) = \lambda_j$,
and finally dividing both sides by $\lambda_j$, we obtain for every $j=1,2,3$:
\begin{equation}
\lambda_j \left | \partial_{v_j} u(z) \right |^2 \leq \int_{B_{1/8} \left ( x_0 \right ) } \left | (z-x) \cdot \nabla u - \alpha u \right |^2 \, \mathrm{d} \mu(x).
\label{BetaEstimateStep1} 
\end{equation} 
Using \eqref{Eigenvalues}, $\lambda_1 \geq \lambda_2 \geq \lambda_3 \geq 0$, the integral of \eqref{BetaEstimateStep1} over $B_{5/4} \left ( x_0 \right ) \backslash B_{3/4} \left ( x_0 \right )$,
and Fubini's Theorem respectively, we get:
\begin{equation}
\begin{aligned}
D_\mu \left ( x_0, 1/8 \right ) \int_{B_{5/4} \left  (x_0 \right ) \backslash B_{3/4} \left ( x_0 \right )} 
\left [ \left | \partial_{v_1} u \right |^2 + \left | \partial_{v_2} u \right |^2  \right ] \, \mathrm{d}z \\
\leq 2 \int_{B_{5/4} \left  (x_0 \right ) \backslash B_{3/4} \left ( x_0 \right )} 
\lambda_2 \left [ \left | \partial_{v_1} u \right |^2 + \left | \partial_{v_2} u \right |^2  \right ] \, \mathrm{d}z \\
\leq 2 \int_{B_{5/4} \left  (x_0 \right ) \backslash B_{3/4} \left ( x_0 \right )} 
\left [ \lambda_1 \left | \partial_{v_1} u \right |^2 + \lambda_2 \left | \partial_{v_2} u \right |^2 + \lambda_3 \left | \partial_{v_3} u \right |^2  \right ] \, \mathrm{d}z \\
\leq 2 \int_{B_{5/4} \left  (x_0 \right ) \backslash B_{3/4} \left ( x_0 \right )}
\int_{B_{1/8} \left ( x_0 \right ) } \left | (z-x) \cdot \nabla u - \alpha u \right |^2 \, \mathrm{d} \mu(x) \, \mathrm{d}z \\
\leq 2 \int_{B_{1/8} \left ( x_0 \right ) } \int_{B_{3/2} \left  (x \right ) \backslash B_{1/2} \left ( x \right )}
 \left | (z-x) \cdot \nabla u - \alpha u \right |^2  \, \mathrm{d}z \, \mathrm{d} \mu(x).
 \label{BetaEstimateStep2}
\end{aligned}
\end{equation}

We can bound the left hand-side of \eqref{BetaEstimateStep2} from below by using the following lower bound:
\begin{equation}
\int_{B_{5/4} \left  (x_0 \right ) \backslash B_{3/4} \left ( x_0 \right )} 
\left [ \left | \partial_{v_1} u \right |^2 + \left | \partial_{v_2} u \right |^2  \right ] \, \mathrm{d}z \geq c \left ( \overline{N} \right ) > 0.
\label{CompactnessLowerBound}
\end{equation}
By $N_{\phi}(0,4) \leq \overline{N}$, \eqref{LogHeight} and the normalization $H_\phi (0,1) = 1$:
$$
\int_{B_1(0)} |\nabla u |^2 \, \mathrm{d}z \leq D_{\phi}(0,4) \leq \overline{N} H_{\phi} (0,4) \leq C \overline{N} H_\phi (0,1) \leq C \overline{N}.
$$
Suppose \eqref{CompactnessLowerBound} is false. Then we can find a sequence of maps with bounded Dirichlet energies, 
normalized $H_\phi (0,1)$, and $u_k^{-1} \{ 0 \} \cap B_{1/8}(0) \neq \emptyset$, as well as 
$x_0^k \in B_{1/8}(0)$, and orthonormal vectors $v_1^k$, $v_2^k$, for which the left-hand side of \eqref{CompactnessLowerBound} is arbitrarily small. 
By the compactness of energy-minimizing maps into $\mathbf{D}_\kappa$ (cf. \cite[Lemma 2.7]{AHL}), we obtain a nontrivial minimizing map $\hat{u}$ with 
$\hat{u}^{-1} \{ 0 \} \cap B_{1/8}(0) \neq \emptyset$,
$p \in \overline{B}_{1/8}(0)$, and orthonormal vectors $\hat{v}_1, \hat{v}_2$, for which the left-hand side of \eqref{CompactnessLowerBound} is $0$.
But then the minimizing map $\hat{u}$ agrees with a single-variable map in some ball $B_{\rho}(q) \subset B_2(0)$. 
This single-variable map defines a minimizing geodesic on $\mathbf{D}_\kappa$, which can be extended indefinitely,
and thus, the single-variable map agreeing with $\hat{u}$ in $B_{\rho}(q)$ extends to a single-variable map in $B_{1/8}(0)$. 
Consequently, by Lemma \ref{UniqueContinuation}, $\hat{u}$ is a single-variable map in $B_{1/8}(0)$. However, $\hat{u}^{-1} \{ 0 \} \cap B_{1/8}(0) \neq \emptyset$.
Since there are no nontrivial minimizing geodesics on $\mathbf{D}_\kappa$ which hit $0 \in \mathbf{D}_\kappa$ at an interior point, we obtain a contradiction. 

Hence, we have:
\begin{equation}
D_\mu \left ( x_0, 1/8 \right ) \leq C \left ( \overline{N} \right ) \int_{B_{1/8} \left ( x_0 \right ) } \int_{B_{3/2} \left  (x \right ) \backslash B_{1/2} \left ( x \right )}
 \left | (z-x) \cdot \nabla u - \alpha u \right |^2  \, \mathrm{d}z \, \mathrm{d} \mu(x).
\label{BetaEstimateStep3}
\end{equation}
The triangle inequality applied to \eqref{BetaEstimateStep3} gives
$D_\mu \left ( x_0, 1/8 \right ) \leq C ( I + II )$,
where:
$$
I = \int_{B_{1/8} \left ( x_0 \right ) } \int_{B_{3/2} \left  (x \right ) \backslash B_{1/2} \left ( x \right )}
\left | (z-x) \cdot \nabla u - N_\phi(x,1) u \right |^2  \, \mathrm{d}z \, \mathrm{d} \mu(x),
$$
$$
II = \int_{B_{1/8} \left ( x_0 \right ) } \int_{B_{3/2} \left  (x \right ) \backslash B_{1/2} \left ( x \right )}
\left ( N_\phi(x,1) - \alpha \right )^2 \left | u \right |^2  \, \mathrm{d}z \, \mathrm{d} \mu(x).
$$

Furthermore, using the triangle inequality and $2ab \leq a^2 + b^2$, we can write $I \leq 2 \left ( I_1 + I_2 \right )$, where:
$$
I_1 = \int_{B_{1/8} \left ( x_0 \right ) } \int_{B_{3/2} \left  (x \right ) \backslash B_{1/2} \left ( x \right )}
\left ( N_\phi (x,1) - N_\phi \left ( x, |z-x| \right ) \right )^2  | u |^2  \, \mathrm{d}z \, \mathrm{d} \mu(x),
$$
$$
I_2 = \int_{B_{1/8} \left ( x_0 \right ) } \int_{B_{3/2} \left  (x \right ) \backslash B_{1/2} \left ( x \right )}
\left | (z-x) \cdot \nabla u - N_\phi \left ( x , |z-x| \right ) u \right |^2  \, \mathrm{d}z \, \mathrm{d} \mu(x).
$$
As $z \in B_{3/2}(x) \backslash B_{1/2}(x)$, $ 1/8 \leq |z-x| \leq 4$. Therefore, by the monotonicity of $N_\phi$:
$$
\left | N_\phi (x,1) - N_\phi \left ( x, |z-x| \right ) \right | \leq N(x,4) - N_\phi(x,1/8) = W_{1/8}^4(x).
$$
Using this bound first, estimating the inner integral in $I_1$ later, and recalling the normalization $H_\phi(0,1)=1$, we get:
\begin{equation*}
I_1 \leq C H_\phi (0,1) \int_{B_{1/8} \left ( x_0 \right ) } \left ( W_{1/8}^4 (x) \right )^2 \, \mathrm{d}\mu(x) 
\leq C \left ( \overline{N} \right )  \int_{B_{1/8} \left ( x_0 \right ) } W_{1/8}^4 (x) \, \mathrm{d}\mu(x).
\end{equation*}
Secondly, the Weiss-type estimate \eqref{QuantWeissEstimate} applied to the inner integral in $I_2$ gives:
\begin{equation*}
I_2 \leq C  \int_{B_{1/8} \left ( x_0 \right ) } W_{1/8}^4 (x) \, \mathrm{d}\mu(x).
\end{equation*}

We can estimate the inner integral in $II$ by making use of $H_\phi(0,1) = 1$, \eqref{LogHeight} and \eqref{LocalHeight}. 
Next we choose $\alpha$ as:
\begin{equation*}
\alpha = \frac{1}{\mu \left ( B_{1/8} \left ( x_0 \right ) \right )} \int_{B_{1/8} \left ( x_0 \right )} N_{\phi} (y,1) \, \mathrm{d} \mu(y), 
\end{equation*}
and apply Jensen's inequality to get:
\begin{equation*}
II \leq \frac{C}{\mu \left ( B_{1/8} \left ( x_0 \right ) \right )} \int_{B_{1/8} \left ( x_0 \right )} \int_{B_{1/8} \left ( x_0 \right )}
\left ( N_\phi (x,1) - N_\phi (y,1) \right )^2 \, \mathrm{d}\mu(y) \, \mathrm{d} \mu(x).
\end{equation*}
Note that $x, y \in B_{1/8} \left ( x_0 \right )$ implies $|x-y| \leq 1/4$.
Hence, we can use the frequency pinching estimate \eqref{PinchingEstimate}, and evaluate the integrals to get:
\begin{equation*}
\begin{aligned}
II  & \leq \frac{C}{\mu \left ( B_{1/8} \left ( x_0 \right ) \right )} \int_{B_{1/8} \left ( x_0 \right )} \int_{B_{1/8} \left ( x_0 \right )}
\left ( W_{1/8}^4 (x) + W_{1/8}^4 (y) \right ) \, \mathrm{d}\mu(y) \, \mathrm{d} \mu(x) \\
    & = 2 C \left ( \overline{N} \right ) \int_{B_{1/8} \left ( x_0 \right )} W_{1/8}^4(x) \, \mathrm{d} \mu(x). 
\end{aligned}
\end{equation*}
The inequality \eqref{BetaEstimateStep3} and the estimates for $I_1$, $I_2$ and $II$ immediately give \eqref{BetaBound}.
\end{proof}

\section{Minkowski content estimate} \label{MeasureSection}

We would like to estimate the Minkowski content of $u^{-1} \{0\}$ in $B_r(x)$ under the assumption that 
the smoothed frequency $N_\phi(y,r)$ is close to $\frac{1}{2\sqrt{\kappa}}$ for every $y \in u^{-1} \{ 0 \} \cap B_r(x)$.
In our Minkowski content estimate, we will crucially rely on the following discrete Reifenberg theorem of Naber and Valtorta in \cite{NV},
which we state in a special case that is relevant to our application.

\begin{theorem}[Naber-Valtorta] \label{NaberValtorta}
Let $ \left \{ B_{s_j} \left ( x_j \right ) \right \}_{j \in J} \subset B_2(0) \subset \mathbf{R}^3$, 
be a collection of pairwise disjoint balls with $\left \{ x_j \right \}_{j \in J} \subset B_1(0)$,
and let:
$$
\mu = \sum_{j \in J} s_j \delta_{x_j}.
$$
There exist absolute positive constants $\delta_0$, $C_0$ such that if for all $B_r(x) \subseteq B_2(0)$ with $x \in B_1(0)$:
\begin{equation}
\frac{1}{r} \int_{B_r(x)} \left ( \int_0^r D_\mu (y,s) \frac{\mathrm{d}s}{s} \right ) \, \mathrm{d}\mu(y) < \delta_0^2,
\label{SmallDeviation}
\end{equation}
then:
\begin{equation}
\mu \left ( B_1(0) \right ) = \sum_{j \in J} s_j \leq C_0.  
\label{NVBound}
\end{equation}
\end{theorem}

We refer to for \cite{NV} for the statement in full generality and its proof. Also compare with the classical Reifenberg theorem, introduced in \cite{Reifenberg60},
which was a key ingredient of the analysis in \cite{AHL}.

\begin{remark} \label{FrequencyRemark}
By \cite[Lemma 4.2]{AHL}, $N \left ( x,r \right ) \geq \frac{1}{2 \sqrt{\kappa}}$ for every non-isolated $x \in u^{-1} \{ 0 \}$ and every $r > 0$.
This lower bound is due to the classification of homogeneous blow-ups of two variables, cf. \cite[Lemma 3.1]{AHL}.
(In fact, by \cite[Corollary 3.6]{AHL}, for every non-isolated $x \in u^{-1} \{ 0 \}$, $N \left ( x, 0^+ \right ) = \frac{1}{2 \sqrt{\kappa}}$, 
if and only if $u$ has a cylindrical tangent map at $x$. Moreover, by \cite[Corollaries 4.4 and 4.5]{AHL}, for any compact $K \subset \Omega$,
this condition is satisfied at all but finitely many $x \in u^{-1} \{0\} \cap K$.)
Also note that \eqref{FreqDerivScale} holds with equality if any only if $u$ is homogeneous on $B_r(x)$. Therefore, we have:
\begin{equation}
N_\phi (x, r) \geq \frac{1}{2 \sqrt{\kappa}},
\label{FrequencyLowerBound}
\end{equation}
for every non-isolated $x \in u^{-1} \{ 0 \}$ and $r > 0$.

We recall that by \cite[Corollary 5.3]{AHL} $u^{-1} \{ 0 \}$ consists locally of a finite union of H\"{o}lder continuous curves (with finitely many crossings) and isolated points.
Since the finite set of isolated points in $u^{-1} \{0\}$ has no effect on the validity of the Minkowski content estimate or the consequent rectifiability result we would like to prove,
we can assume without losing generality that $u^{-1} \{0\}$ consists locally of a finite union of H\"{o}lder continuous curves with finitely many crossings.
\end{remark}

The following lemma corresponds to Theorem 2.5, Proposition 7.2 and Proposition 7.3 in \cite{DLMSV} combined.
While we are following the strategy of \cite{DLMSV}, the proof is simpler, due to the strong assumption \eqref{SmallnessAssumption},
which we are able to verify in the applications of Lemma \ref{ConditionalEstimate} in the proof of Theorem \ref{LocalFiniteLength}. 

\begin{lemma} \label{ConditionalEstimate}
Let $u \, : \, \Omega \to \mathbf{D}_\kappa$ be an energy-minimizing map and $B_{2r}(x) \subset \Omega$. 
There exist $\delta = \delta \left ( \overline{N} \right )$ and an absolute constant $C$ such that:
\begin{equation}
\sup_{y \in u^{-1} \{ 0 \} \cap B_{r}(x)} N_\phi (y,r) < \frac{1}{2 \sqrt{\kappa}} + \delta
\label{SmallnessAssumption} 
\end{equation}
implies:
\begin{equation}
\left | B_\rho \left ( u^{-1} \{ 0 \} \cap B_r(x) \right ) \right | \leq C r \rho^2, \quad \forall \rho \in (0,r]. 
\label{ConditionalMinkowski} 
\end{equation}
\end{lemma}

\begin{proof}
Our goal is to construct for any given $\rho  \in (0,r]$, a finite cover $\left \{ B_{\rho} \left ( x_i \right )  \right \}$ of $u^{-1} \{ 0 \} \cap B_r(x)$ such that
$M \rho \leq C r/8$, where $M$ is the number of balls in the cover and $C$ is an absolute constant. Once we construct such a cover, we obtain:
\begin{equation}
\left | B_\rho \left ( u^{-1} \{ 0 \} \cap B_r(x) \right ) \right | \leq \left | \cup_i B_{2 \rho} \left ( x_i \right ) \right | \leq M \left ( 2 \rho \right )^3 \leq C r \rho^2.
\label{FinalStroke}
\end{equation}
which is \eqref{ConditionalMinkowski}. Hence, we fix an arbitrary $\rho \in (0,r)$, as the case $\rho = r$ holds trivially.
Firstly, we will build a finite cover $\left \{ B_{\rho} \left ( x_i \right )  \right \}$ by induction, then we will verify $M \rho \leq C r/8$ by another induction argument.
For brevity, we introduce the notation $\mathcal{Z}_r(x) = u^{-1} \{ 0 \} \cap B_r(x)$.

{\bf{Construction:}} We begin with the cover $\mathcal{F}_0 = \left \{ B_r(x) \right \}$. 
Since by Remark \ref{FrequencyRemark} we assume that $\mathcal{Z}_r(x)$ has no isolated points, 
either $\mathcal{Z}_r(x) \subset B_{r/10} (y)$ for some $y \in \mathcal{Z}_r(x)$, 
or there are $y_1, y_2 \in \mathcal{Z}_r(x)$ such that $\mathrm{dist} \left ( y_1, y_2 \right ) \geq r/10$.
If the first case holds, then we proceed our construction with $\left \{ B_{r/10}(y) \right \}$ instead of $\left \{ B_{r}(x) \right \}$.
Otherwise, we let $V$ be the line such that $y_1, y_2 \in V$. 
For $\delta = \delta \left ( \overline{N} \right )$ small enough, $\mathcal{Z}_r(x)$ is contained in $B_{r/20}(V)$ by \eqref{HalfReifenberg} in Lemma \ref{FrequencyControl}.
Cover $\mathcal{Z}_r(x) \cap B_{r/20}(V)$ with balls of radius $r/10$ that satisfy the following properties:
\begin{enumerate}[(i)]
 \item the centers $x_i \in \mathcal{Z}_r(x) \cap V$,
 \item the concentric balls of radius $r/50$ are pairwise disjoint. 
\end{enumerate}
We denote such a cover as $\left \{ B_{r/10} \left ( x_i \right ) \right \}_{i \in I_1}$.
Hence, $\mathcal{F}_1$ is either $\left \{ B_{r/10}(y) \right \}$ or $\left \{ B_{r/10} \left ( x_i \right ) \right \}_{i \in I_1}$.

We apply this argument inductively to each ball in $\mathcal{F}_k$, $k \geq 1$, while maintaining the Vitali property.
In particular, for each $B_{r/10^k}(z) \in \mathcal{F}_k$: 
\begin{enumerate}[(a)]
 \item either $\mathcal{Z}_r(x) \cap B_{r/10^k}(z) \subset B_{r/10^{k+1}} \left ( \tilde{z} \right )$ for some $\tilde{z} \in \mathcal{Z}_r(x) \cap B_{r/10^k}(z)$, \label{CaseA}
 \item or there are $z_1, z_2 \in \mathcal{Z}_r(x) \cap B_{r/10^k}(z)$ such that $\mathrm{dist} \left  ( z_1, z_2 \right ) \geq r/10^{k+1}$. \label{CaseB}
\end{enumerate}
If \eqref{CaseA} holds, we include $B_{r/10^{k+1}} \left ( \tilde{z} \right ) $ in $\mathcal{F}_{k+1}$. 
For each $B_{r/10^k}(z) \in \mathcal{F}_k$ satisfying \eqref{CaseB}, we denote by $V_z$ the line containing $z_1$ and $z_2$.
As in the first step, for $\delta = \delta \left ( \overline{N} \right )$ small enough, $\mathcal{Z}_r(x) \cap B_{r/10^k}(z) \subset B_{r/ \left ( 2 \cdot 10^{k+1} \right )}  \left (V_z \right )$.
We cover the set:
$$
\mathcal{Z}_r(x) \cap \bigcup_z \left ( B_{r / 2 \cdot 10^{k+1}} \left ( V_z \right ) \cap B_{r/10^k}(z) \right ),
$$
where the union is over those centers $z$ that satisfy \eqref{CaseB} above,
with balls of radius $r/10^{k+1}$ with the following properties:
\begin{enumerate}[(i)]
 \item the centers $x_j$ lie on $\mathcal{Z}_r(x) \cap \left ( \cup_z B_{r/10^k}(z) \cap V_z \right )$,
 \item the concentric balls of radius $r/\left ( 5 \cdot 10^{k+1} \right )$ are pairwise disjoint.
\end{enumerate}
From these balls, we add to $\mathcal{F}_{k+1}$ those which do not intersect
$B_{r/ \left ( 5 \cdot 10^k \right )}(y)$, where $B_{r/ 10^k}(y) \in \mathcal{F}_k$ such that $\mathcal{Z}_r(x) \cap B_{r/ 10^k}(y) \subset B_{r/10^{k+1}} \left ( \tilde{y} \right )$
for some $\tilde{y} \in \mathcal{Z}_r(x) \cap B_{r/ 10^k}(y)$. We claim that despite the omission of balls $B_{r/10^{k+1}}(z)$ failing to satisfy this condition, 
$\mathcal{F}_{k+1}$ is still a cover of $\mathcal{Z}_r(x)$. Indeed, in this case, we observe that $B_{r/10^{k+1}}(z) \subset  B_{r/10^k}(y)$, and 
$ \mathcal{Z}_r(x) \cap B_{r/10^k}(y)  \subset B_{r/10^{k+1}} \left ( \tilde{y} \right )$. Since $B_{r/10^{k+1}} \left ( \tilde{y} \right ) \in \mathcal{F}_{k+1}$, the claim follows.

We enumerate these finite covers as $\mathcal{F}_k = \left \{ B_{r/10^{k}} \left ( x_i \right ) \right \}_{i \in I_k}$, for any integer $k \geq 1$.
In particular, we obtain $\mathcal{F}_\tau = \left \{ B_{r/10^{\tau}} \left ( x_i \right ) \right \}_{i \in I_\tau}$ 
for $\tau = \left \lceil{\log_{10} \left ( r/\rho \right )}\right \rceil $,
which is finer cover than $ \left \{ B_\rho \left ( x_i \right ) \right \}_{i \in I_\tau}$. We remark that we have already imposed a smallness assumption on $\delta$ 
(depending on $\overline{N}$ only) in our construction. Our next task is to show that by shrinking $\delta$ further, if necessary, we can ensure that:
\begin{equation}
M \frac{r}{10^\tau} = \sum_{i \in I_\tau} \frac{r}{10^\tau} \leq \frac{C}{80} r,
\label{CardinalityBound}
\end{equation}
where $M$ is the cardinality of the index set $I_\tau$, and $C$ is the absolute constant in \eqref{ConditionalMinkowski}. 
Since $\rho/10 \leq r/10^\tau$, \eqref{CardinalityBound} immediately implies the last inequality in \eqref{FinalStroke}, and consequently \eqref{ConditionalMinkowski} as well.

{\bf{Cardinality bound:}} We let $\sigma = r/ \left ( 5 \cdot 10^\tau \right)$ and introduce the measures: $\mu = \sigma \sum_{i \in I_\tau} \delta_{x_i}$. 
Hence, \eqref{CardinalityBound} would follow from:
\begin{equation}
\mu \left ( B_r(x) \right ) \leq \tilde{C}_0 r, 
\label{MeasureBound} 
\end{equation}
once we choose $C = 400 \tilde{C}_0$. Here $\tilde{C}_0$ is to be determined.

If $\sigma \geq r/128$, then the number of balls of radius $\sigma$ needed to cover $B_r(x)$ is bounded by an absolute constant. 
In this case \eqref{MeasureBound} follows from the fact that $\left \{ B_\sigma \left ( x_i \right ) \right \}_{i \in I_\tau}$ are pairwise disjoint.
Otherwise, we set $A$ to be the largest positive integer satisfying $2^A \sigma < r/64$. 
We will prove by induction that there exists an absolute constant $C_0$ such that:
\begin{equation}
\mu \left ( B_{2^j \sigma}(z) \right ) \leq C_{0} 2^j \sigma,
\label{InductionBound}
\end{equation}
for every $z \in B_r(x)$ and $j = 0,1,2, ..., A$.
We note that by the choice of $A$, $2^A \sigma \geq r/128$. 
In particular, the number of balls of radius $2^A \sigma$ needed to cover $B_r(x)$ is bounded by an absolute constant. 
Hence, \eqref{InductionBound} with $j = A$ would imply \eqref{MeasureBound}.

Firstly, we note that by construction $\left \{ B_\sigma \left ( x_i \right ) \right \}_{i \in I_\tau}$ is a collection of pairwise disjoint balls in $B_{2r}(x)$ 
with $x_i \in B_r(x)$ for every $i \in I_\tau$. As a result, $\mu \left ( B_{\sigma}(z) \right ) \leq 8 \sigma$, and as long as $C_{0} \geq 8$,
\eqref{InductionBound} holds in the base case $j = 0$.

We assume that \eqref{InductionBound} holds for every $z \in B_r(x)$ for some $j < A$, and set $\tilde{\sigma} = 2^j \sigma$. 
There remains to show \eqref{InductionBound} for every $z \in B_r(x)$ for $j+1$. That is:
\begin{equation}
\mu \left ( B_{2 \tilde{\sigma}}(z) \right ) \leq C_{0} \left ( 2 \tilde{\sigma} \right ), \quad \forall z \in B_r(x).
\label{FinalEstimate}
\end{equation}

We observe that since the number of balls of radius $\tilde{\sigma}$ required to cover $B_{2 \tilde{\sigma}}(z)$ is bounded by an absolute constant $\hat{C}$,
covering and the induction hypothesis together imply the coarse bound:
\begin{equation}
\mu \left ( B_{2 \tilde{\sigma}}(z) \right ) \leq \frac{\hat{C}}{2} C_{0} \left ( 2 \tilde{\sigma} \right ).
\label{CoarseBound}
\end{equation}
However, in order to obtain \eqref{FinalEstimate}, we need to invoke Theorem \ref{NaberValtorta}. 
We fix an arbitrary $z \in B_r(x)$ and set $\overline{\mu} = \mu \llcorner B_{2 \tilde{\sigma}}(z)$. 
Once we verify the hypothesis that for every $y \in B_{2 \tilde{\sigma}}(z)$ and every $t \in \left ( 0, 2 \tilde{\sigma} \right )$:
\begin{equation}
I(t) = \int_{B_t(y)} \left ( \int_0^t D_{\overline{\mu}} (\zeta,s) \, \frac{\mathrm{d}s}{s} \right ) \, \mathrm{d} \overline{\mu}(\zeta) < \delta_0^2 t,
\label{MainHypothesis}
\end{equation}
after an appropriate scaling, Theorem \ref{NaberValtorta} gives $\overline{\mu} \left ( B_{2 \tilde{\sigma}}(z) \right ) \leq C_{0} \left ( 2 \tilde{\sigma} \right )$,
for $C_0$ as in \eqref{NVBound}, which is \eqref{FinalEstimate}.

Finally, we will combine \eqref{BetaBound}, \eqref{CoarseBound} and Fubini's theorem to verify \eqref{MainHypothesis} for any fixed $t \leq 2 \tilde{\sigma}$. 
Setting:
\begin{equation*}
\overline{W}_s \left ( \zeta \right ) =  
\begin{cases}
N_\phi \left ( \zeta, 32 s \right ) - N_\phi \left ( \zeta , s \right ), & \mathrm{if} \; s > \sigma, \\
0, \quad & \mathrm{otherwise},
\end{cases}
\end{equation*}
we note that for all $i \in I_\tau$ and all $s \in (0,1)$:
\begin{equation}
D_{\overline{\mu}} \left ( x_i , s \right ) \leq C \left ( \overline{N} \right ) \frac{1}{s} \int_{B_s \left ( x_i \right ) } \overline{W}_s( \xi ) \mathrm{d} \overline{\mu} (\xi),
\label{PointwiseBound}
\end{equation}
since this is exactly \eqref{BetaBound} for $s > \sigma$, and otherwise $ \mathrm{spt} ( \mu ) \cap B_s \left ( x_i \right ) = \left \{ x_i \right \}$, which implies that 
both sides of \eqref{PointwiseBound} are $0$. 

Using \eqref{PointwiseBound} and Fubini's Theorem, we get:
\begin{equation*}
\begin{aligned}
I(t) & \leq C \int_0^t \frac{1}{s^2} \int_{B_t(y)} \int_{B_s(\zeta)} \overline{W}_s ( \xi ) \, \mathrm{d} \overline{\mu} (\xi) \, \mathrm{d}s \, \mathrm{d} \overline{\mu} ( \zeta) \\
     & = C \int_0^t \frac{1}{s^2} \int_{B_t(y)} \int_{B_s(\zeta)} \overline{W}_s ( \xi ) \, \mathrm{d} \overline{\mu} ( \xi ) \, \mathrm{d} \overline{\mu} ( \zeta ) \, \mathrm{d}s.
\end{aligned}
\end{equation*}
By the fact that $\overline{\mu}$ is supported on $B_{2 \tilde{\sigma}}(z)$ and Fubini's Theorem once again:
\begin{equation*}
I(t) \leq C \int_0^t \frac{1}{s^2} 
\int_{B_{t+s}(y) \cap B_{2 \tilde{\sigma}(z)}} \overline{W}_s ( \xi ) \int_{B_s(\xi) \cap B_{2 \tilde{\sigma}(z)}} \, \mathrm{d} \mu(\zeta) \, \mathrm{d}\mu(\xi) \, \mathrm{d}s.
\end{equation*}
Using the coarse bound \eqref{CoarseBound} to estimate the inner most integral, and applying Fubini's theorem one last time, we get:
\begin{equation}
\begin{aligned}
I(t) & \leq C \int_0^t \int_{B_{t+s}(y) \cap B_{2 \tilde{\sigma}(z)}} \overline{W}_s ( \xi ) \, \mathrm{d} \mu \, \frac{\mathrm{d}s}{s}
     & \leq C \int_{B_{2t}(y)} \int_0^t \overline{W}_s ( \xi ) \, \frac{\mathrm{d}s}{s} \mathrm{d} \overline{\mu}. 
 \label{NVHStep1}
\end{aligned}
\end{equation}

Next we estimate the inner integral for an arbitrary $\xi = x_i \in B_{2 \tilde{\sigma}}(z) \cap \mathcal{Z}_r(x)$ that we fix. 
Let $\tilde{A}$ be the smallest integer that $2^{\tilde{A}} \sigma > t$. Then since $j+1 \leq A$:
\begin{equation*}
2^{\tilde{A}-1} \leq t \leq 2 \tilde{\sigma} = 2^{j+1} \sigma \leq 2^A \sigma < \frac{r}{64},
\end{equation*}
and therefore:
\begin{equation}
2^{\tilde{A}} \sigma > t, \quad \mathrm{and} \; 2^{\tilde{A}+5} < r.
\label{ThingsFit}
\end{equation}
Using the monotonicity of \eqref{FreqDerivScale}, \eqref{ThingsFit} and \eqref{SmallnessAssumption} we estimate:
\begin{equation}
\begin{aligned}
\int_0^t \overline{W}_s ( \xi ) \, \frac{\mathrm{d}s}{s} 
& = \int_\sigma^t \left [ N_\phi \left ( x_i, 32s \right  ) - N_\phi \left ( x_i, s \right  ) \right ] \, \frac{\mathrm{d}s}{s} \\
& \leq \sum_{l = 0}^{\tilde{A}-1} \int_{2^l \sigma}^{2^{l+1} \sigma}  \left [ N_\phi \left ( x_i, 32 s \right  ) - N_\phi \left ( x_i, s \right  ) \right ] \, \frac{\mathrm{d}s}{s} \\
& \leq \sum_{l = 0}^{\tilde{A}-1} \left [ N_\phi \left ( x_i, 2^5 \cdot 2^{l+1} \sigma \right  ) - N_\phi \left ( x_i, 2^l \sigma \right  ) \right ] 
\int_{2^l \sigma}^{2^{l+1} \sigma} \, \frac{\mathrm{d}s}{s} \\
& = \log 2 \sum_{l=0}^{\tilde{A}-1} \left [ N_\phi \left ( x_i, 2^{l+6} \sigma \right  ) - N_\phi \left ( x_i, 2^l \sigma \right  ) \right ] \\
& = \log 2 \sum_{m=0}^5 \sum_{l=0}^{\tilde{A}-1} \left [ N_\phi \left ( x_i, 2^{l+m+1} \sigma \right  ) - N_\phi \left ( x_i, 2^{l+m} \sigma \right  ) \right ] \\
& =  \log 2 \sum_{m=0}^5 \left [ N_\phi \left ( x_i, 2^{\tilde{A}+m} \sigma \right  ) - N_\phi \left ( x_i, 2^{m} \sigma \right  ) \right ] \\
& \leq 6 \log 2 \left [ N_\phi \left ( x_i, r \right  ) - N_\phi \left ( x_i, \sigma \right  ) \right ] \\
& \leq 6 \log 2 \left [ \frac{1}{2 \sqrt{\kappa}} + \delta - \frac{1}{2\sqrt{\kappa}} \right ] \leq \left  ( \log 64 \right ) \delta. 
\end{aligned}
\label{PenultimateBound}
\end{equation}

Finally, covering $B_{2t}(y) \subset B_{20 \tilde{\sigma}}(z)$ by balls of radius $\tilde{\sigma}$, the number of which can be bounded by an absolute constant,
and using the coarse bound \eqref{CoarseBound}, we can estimate: 
\begin{equation}
\overline{\mu} \left ( B_{2t}(y) \right ) \leq \mu \left ( B_{2t}(y) \right ) \leq \hat{C} t.
\label{CoarseApplied}
\end{equation}

Combining \eqref{NVHStep1}, \eqref{PenultimateBound} and \eqref{CoarseApplied}, we obtain:
\begin{equation}
\int_{B_t(y)} \left ( \int_0^t D_{\overline{\mu}} (\zeta,s) \, \frac{\mathrm{d}s}{s} \right ) \, \mathrm{d} \overline{\mu}(\zeta) < C \left ( \overline{N} \right ) \delta t. 
\label{UltimateBound}
\end{equation}
Hence, shrinking $\delta = \delta \left ( \overline{N} \right )$ further, if necessary, to ensure $\delta \leq \delta_0^2 / C \left ( \overline{N} \right )$, we obtain \eqref{MainHypothesis}.
Thus, the proof of \eqref{FinalEstimate} is complete. 
\end{proof}

\section{Topological Structure} \label{TopologicalSection}

In this section we restate several important lemmas from \cite{AHL} for the reader's convenience.
We will heavily rely on these lemmas in the proof of Theorem \ref{LocalFiniteLength} in Section \ref{FinalSection}.
Firstly, we recall the characterization of the defect set of homogeneous maps $v: \mathbf{R}^3 \to \mathbf{D}_k$ minimizing the Dirichlet energy in any bounded subset of $\mathbf{R}^3$.

\begin{lemma} \label{3DHomogeneous}
There exist positive constants $N_0$, $d_0$, $C$ so that for any non-constant homogeneous map $v: \mathbf{R}^3 \to \mathbf{D}_k$ minimizing the Dirichlet energy,
$v^{-1}\{ 0 \} \cap \mathbf{S}^2$ consists of $2m$ points separated by distances at least $d_0$, where $2m \leq K_0$,
where both $d_0$ and $K_0$ depend on $N \left ( 0; 0^+ \right )$, namely the vanishing order of $v$ at $0$. 
Near each $a \in v^{-1}\{ 0 \} \cap \mathbf{S}^2$, the following asymptotic estimate holds:
\begin{equation}
\left | v(x) - w_a \circ p_a (x-a) \right | \leq C \left ( |x-a|^{1/2 \sqrt{k}} \right ),
\label{asymptoticEstimate} 
\end{equation}
for some two-dimensional minimizer $w_a$ and orthogonal projection $p_a \, : \, \mathbf{R}^3 \to \mathbf{R}^2$ with $p_a(a) = 0$. 
Furthermore, for $\epsilon > 0$,
there exist $\beta = \beta \left ( \epsilon, N \left (0; 0^+ \right ) \right )> 0$ and $\gamma = \gamma \left ( \epsilon, N \left ( 0; 0^+ \right ) \right ) > 0$ such that:
\begin{equation}
N_\phi \left ( b ; r \right ) \leq \frac{1}{2 \sqrt{k} } + \epsilon,
\label{frequencyperturbation}
\end{equation}
whenever $ \left | a - \frac{b}{|b|} \right | < \beta$ and
$r \in (0, |b| \gamma ]$ for some $a \in v^{-1} \{0\} \cap \mathbf{S}^2$.  
\end{lemma}

The proof of Lemma \ref{3DHomogeneous} is based on successive compactness arguments, as well as a topological argument.
We remark that the original statement in \cite[Lemma 3.4]{AHL} is in terms of the classical frequency function $N$. 
However, since \eqref{frequencyperturbation} is proved via a compactness argument, 
combining the proof of \cite[Lemma 3.4]{AHL} with Lemma \ref{UniqueContinuation} immediately yields the claim.

Before stating the next lemma, we need to recall the definition of isolated and non-isolated defects of maps $u \, : \, \Omega \to \mathbf{D}_k$ minimizing the Dirichlet energy.

\begin{definition}
We denote $u^{-1} \{ 0 \} = \mathcal{Z}_0 \cup \mathcal{Z}_1$, where:
\begin{equation}
\mathcal{Z}_0 = \left \{ a \in u^{-1} \{ 0 \} \, : \, u_{\infty}^{-1} \{ 0 \} \cap \mathbf{S}^2 = \emptyset \; \mathrm{for} \; \mathrm{every} \; \mathrm{tangent} \; \mathrm{map}
\; u_{\infty} \; \mathrm{at} \; a \right \},
\end{equation}
and
$\mathcal{Z}_1 = u^{-1} \{ 0 \} \backslash \mathcal{Z}_0$. 
\end{definition}

The main application of Lemma \ref{3DHomogeneous} is the following result on the defect set of energy-minimizing maps, which are not necessarily homogeneous.

\begin{lemma} \label{nonisolated2}
For any map 
$u \, : \, B_{64}(0) \to \mathbf{D}_k$ minimizing the Dirichlet energy
with $0 \in \mathcal{Z}_1$ 
and for every $\epsilon > 0$, there exists an
$R = R \left ( \epsilon, N \left ( 0; 0^+ \right ) \right ) > 0$ and
positive $2m \leq K_0 = 
K_0 \left ( N \left ( 0; 0^+ \right ) \right ) $, so that
for each $r \in (0, R ]$ there is a corresponding
homogeneous energy minimizing map $v_r$ such that
$v_r^{-1} \{ 0  \} \cap \mathbf{S}^2$ has
exactly $2m$ points, and for
\begin{equation*}
u_r(x) = \left ( \fint_{ B_r(0)} |u|^2 \, \mathrm{d}A \right )^{-1/2} u(rx),
\end{equation*}
there holds:
\begin{equation}
\left \| u_r - v_r \right \|_{H^1 \left ( B_1(0) \right ) }  \leq \epsilon.
\label{nonisolatedE2a}
\end{equation}
Moreover, the following inclusions hold:
\begin{equation}
\left ( \overline{B_r(0)} \backslash B_{r/2}(0) \right ) 
\cap u^{-1} \{0\} \subset
\left \{ x \; : \; \mathrm{dist} \left (x, v_r^{-1} \{ 0 \} \right )
< r \epsilon \right \},
\label{nonisolatedE2b}
\end{equation}
and
\begin{equation}
\left ( \overline{B_r(0)} \backslash B_{r/2}(0) \right )
\cap v_r^{-1} \{0\} \subset
\left \{ x \; : \; \mathrm{dist} \left (x, \mathcal{Z}_1 \right )
< r \epsilon \right \}.
\label{nonisolatedE2c}
\end{equation} 
\end{lemma}

The proofs of \eqref{nonisolatedE2a} and \eqref{nonisolatedE2b} are based on compactness, while \eqref{nonisolatedE2c} is proved via a topological argument.
We refer to \cite[Lemma 4.3]{AHL} for the details. An important consequence of Lemma \ref{nonisolated2} is the following corollary:

\begin{corollary} \label{discretezeros}
For any map $u \, : \, \Omega \to \mathbf{D}_k$ minimizing the Dirichlet energy and any compact $K \subset \Omega$, $\mathcal{Z}_0 \cap K$ is a discrete set.
\end{corollary}

We refer to \cite[Corollary 4.4]{AHL} for the proof based on Lemma \ref{nonisolated2} and a contradiction argument, and recall the following conditional structure
result for the defect set of maps minimizing the Dirichlet energy.

\begin{lemma} \label{StructureLemma}
For every $\epsilon > 0$, there exists a $\delta_0 = \delta_0 ( \epsilon )$ such that if $u \, : \, B_{64}(0) \to \mathbf{D}_k$ is 
a map minimizing the Dirichlet energy and satisfying:
\begin{equation*}
\left ( \overline{B_1(0)} \backslash B_{1/2}(0) \right ) \cap u^{-1} \{0\} \neq \emptyset, \quad \mathrm{and}
\end{equation*}
\begin{equation}
N_\phi(0;2) < \frac{1}{2 \sqrt{k}} + \delta_0,
\label{PerturbativeHypothesis}
\end{equation}
then for each $b \in B_1(0) \cap \mathcal{Z}_1$, $0<r \leq 1/2$, there exists $L_r^b$, a line passing through $b$, such that the following hold:
\begin{equation}
\overline{B_r(b)} \cap u^{-1} \{ 0\} \subset \left \{ x \, : \, \mathrm{dist} \left ( x, \mathcal{Z}_1 \right ) < r \epsilon \right \}, \quad \mathrm{and}
\label{StructureInclusion1}
\end{equation}
\begin{equation}
\overline{B_r(b)} \cap L_r^b \subset \left \{ x \, : \, \mathrm{dist} \left ( x, \mathcal{Z}_1 \right ) < r \epsilon \right \}.
\label{StructureInclusion2}
\end{equation}
Furthermore, there holds
\begin{equation}
B_{1/2}(0) \cap \mathcal{Z}_1 \subset \Gamma \subset B_1(0) \cap \mathcal{Z}_1,
\label{StructureInclusion3}
\end{equation}
for a single embedded H\"older continuous arc $\Gamma$. 
\end{lemma} 

While we refer to \cite[Lemma 5.1]{AHL} for a proof, several remarks are in order. 
As in Lemma \ref{3DHomogeneous}, replacing the classical frequency function in the original statement in \cite[Lemma 5.1]{AHL} with the smoothed frequency
is not an issue, as \eqref{PerturbativeHypothesis} is used in a compactness argument.
We also note that \eqref{StructureInclusion1} and \eqref{StructureInclusion2} are analogous to \eqref{nonisolatedE2b} and \eqref{nonisolatedE2c} in Lemma \ref{nonisolated2}.
However, \eqref{StructureInclusion1} and \eqref{StructureInclusion2} hold at every scale $r \in (0,1/2]$, 
which is proved by an iterative argument and a special case of Lemma \ref{3DHomogeneous}.
Finally, \eqref{StructureInclusion3} follows from \eqref{StructureInclusion1} and \eqref{StructureInclusion2} via Reifenberg's Topological Disk Theorem, cf. \cite{Reifenberg60}.

\section{Proof of the Main Result} \label{FinalSection}

Firstly, we prove Theorem \ref{LocalFiniteLength}.

\begin{proof}[Proof of Theorem \ref{LocalFiniteLength}]
If $b$ is an isolated zero in $u^{-1} \{ 0 \}$, there is nothing to prove. 
Without loss of generality we assume that $b = 0$ is non-isolated, and the set of isolated zeros $\mathcal{Z}_0 = \emptyset$.
The latter is possible, since $\mathcal{Z}_0$ is a discrete set by Corollary \ref{discretezeros}. 
Theorem \ref{LocalFiniteLength} differs from \cite[Theorem 5.2]{AHL} by the additional measure estimate $\mathcal{H}^1 \left ( \Gamma_i \right ) < + \infty$ for each $i = 1, 2, ..., 2m$.
For the sake of completeness we recall the proof of \cite[Theorem 5.2]{AHL} first.

{\bf{Step 1:}}
For $\epsilon > 0$ to be determined, there exist an $R = R \left ( \epsilon, N_\phi \left ( 0, 0^+ \right ) \right ) > 0$ and an approximating homogeneous minimizer $v_r$
for each $r \in (0,R]$ by Lemma \ref{nonisolated2}. 
By \eqref{nonisolatedE2c}, for each $a \in v_r^{-1} \{0\} \cap \mathbf{S}^2$, there is at least a point: 
$$ b_{a,r} \in \partial B_{\frac{3r}{4}}(0) \cap B_{\epsilon r} \left ( \frac{3a}{4} \right ) \cap u^{-1} \{ 0 \}. $$
Before applying Lemma \ref{StructureLemma}, 
we firstly choose $\delta = \delta \left ( d_0 /16 \right )$ as in Lemma \ref{StructureLemma}, where $d_0$ is as in Lemma \ref{3DHomogeneous}, 
(and  $d_0$ hence depends on $N_\phi \left ( 0; 0^+ \right )$ only.)
Secondly, we choose $\beta = \beta \left ( d_0 /2, N_\phi \left ( 0, 0^+ \right ) \right )$ 
and $\gamma = \gamma \left (d_0 /2, N_\phi \left ( 0, 0^+ \right ) \right )$ as in Lemma \ref{3DHomogeneous}.
Finally, we choose $R = R \left ( \epsilon, N_\phi \left ( 0, 0^+ \right ) \right )$ for $\epsilon$ yet to be determined.
Hence, we can update all our parameters successively, based on our choice of $\epsilon$.
For such parameters, we obtain:
\begin{equation}
N_\phi \left ( b_{a,r}, 2 \gamma r  \right ) < N_{v_r, \phi} \left ( b_{a,r}, 2 \gamma r \right ) + \delta/2 <  \frac{1}{2 \sqrt{\kappa}} + \delta,
\label{FirstStepEstimate}
\end{equation}
where the first inequality follows from \eqref{nonisolatedE2a} for $\epsilon$ chosen sufficiently small, and the second inequality follows from Lemma \ref{3DHomogeneous},
as $\beta$ and $\gamma$ have been chosen sufficiently small with respect to $\delta$. 

{\bf{Step 2:}} Shrinking $\epsilon$ further so that it is much smaller than $\gamma/2$, we claim that the map $u_{a,r} = u \left ( b_{a,r} + \gamma r x \right )$ satisfies 
the hypothesis of Lemma \ref{StructureLemma}. 
Since we have already obtained \eqref{FirstStepEstimate}, there remains to check that:
\begin{equation}
\left ( \overline{ B_{r\gamma} \left ( b_{a,r} \right ) } \backslash B_{\frac{r\gamma}{2}} \left ( b_{a,r} \right ) \right ) \cap u^{-1} \{0\} \neq \emptyset.
\label{StepTwoVerification}
\end{equation}
This follows from the inclusion \eqref{nonisolatedE2c} and the observation that $B_{r \epsilon}(p)$ is contained in the annulus
$\overline{ B_{r \gamma} \left ( b_{a,r} \right ) } \backslash B_{\frac{r \gamma}{2}} \left ( b_{a,r} \right )$ for some $p \in v_r^{-1} \{0\}$,
since $\left | b_{a,r} - \frac{3}{4}a \right | < r \epsilon$.
Hence, applying Lemma \ref{StructureLemma}, by the choice $\delta = \delta \left ( d_0 /16 \right )$, we conclude:
\begin{equation*}
\overline{B_{\frac{r \gamma}{2}} \left ( b_{a,r} \right )} \cap u^{-1} \{0\} \subset \Gamma_{a,r} \subset B_{r\gamma} \left ( b_{a,r} \right ) \cap u^{-1} \{ 0 \}
\subset \left \{ x \, : \, \mathrm{dist} \left ( x, L_{a,r} \right ) < d_0 r \gamma /16 \right \},
\end{equation*}
for some embedded H\"older continuous arc $\Gamma_{a,r}$ and some line $L_{a,r}$ passing through $b_{a,r}$.

{\bf{Step 3:}} Once again recalling that $ \left | b_{a,r} - \frac{3}{4}a \right | < r \epsilon$, 
where $\epsilon$ has been chosen to be much smaller than $\gamma/2$, we obtain:
\begin{equation*}
\left ( \overline{B_{\frac{3}{4}r + \frac{1}{4}r\gamma }(0)} \backslash B_{\frac{3}{4}r - \frac{1}{4}r\gamma }(0) \right ) 
\cap \left \{ x \, : \, \mathrm{dist} \left ( x, v_r^{-1} \{ 0 \} \right ) < \epsilon r \right \} \subset \overline{B_{\frac{r\gamma}{2}} \left ( b_{a,r} \right )},
\end{equation*}
for each $a \in v_r^{-1} \{ 0 \} \cap \mathbf{S}^2$ and the corresponding $b_{a,r}$. Consequently:
\begin{equation*}
\left ( \overline{B_s(0)} \backslash B_{\lambda s}(0) \right ) \cap u^{-1} \{ 0 \} \subset \bigcup_{a \in v_r^{-1} \{ 0 \} } \Gamma_{a,r} \subset u^{-1} \{ 0 \},
\end{equation*}
where $s = \frac{3}{4}r + \frac{1}{4} r \gamma $ and $ \lambda = 1 - \frac{2 \gamma}{3+\gamma}$.
By the inclusion \eqref{nonisolatedE2c}, each arc $\Gamma_{a,r}$ intersects both the outer sphere $\partial B_s(0)$ and the inner sphere $\partial B_{\gamma s}(0)$.
Arguing as in Lemma \ref{nonisolated2} we infer that when $r$ is sufficiently small, $\Gamma_{a,r}$ overlaps with $\Gamma_{\bar{a}, \lambda r }$, when $\bar{a}$ is the nearest
point to $a$ in $v_{\lambda r}^{-1} \{ 0 \} \cap \mathbf{S}^2$. We finally note that $\Gamma_{a,r} \cup \Gamma_{\bar{a}, \lambda r}$ is clearly also a H\"older continuous arc.

{\bf{Step 4:}} Beginning with $r = R$, we iterate the argument in Steps 1, 2 and 3 with $r = R, \lambda R, \lambda^2 R$, ...  
Consequently, we construct chains of overlapping arcs starting with $\Gamma_{a,R}$ for each
$a \in v^{-1}_R \{ 0 \} \cap \mathbf{S}^2$, which we enumerate as $\Gamma_1$, $\Gamma_2$, ..., $\Gamma_{2m}$. 
Finally, we define $\mathcal{O}$ to be $B_{\frac{3R}{4}}(0)$. 

{\bf{Measure estimate:}}
There remains to prove that $\mathcal{H}^1 \left ( \Gamma_i \right ) < + \infty$ for each $i = 1, 2, ..., 2m$. 
The idea is to shrink $\delta$ in the above construction to check \eqref{SmallnessAssumption} in $B_{\tilde{r}} \left (b_{a,R} \right )$ for some $\tilde{r}$, and
estimate $\mathcal{H}^1 \left ( \Gamma_i \right )$ by an iteration similar to the one above. 
Updating the choice of $\delta$ in Step 1 with respect to Lemma \ref{ConditionalEstimate}, for $\gamma$ and $R$ as above, we have:
$$ N_\phi \left ( b_{a,R}, 2 \gamma R  \right ) < \frac{1}{2 \sqrt{\kappa}} + \delta, $$ 
for $b_{a,R} \in \Gamma_j$ for some $j \in \{1,2, ..., 2m \}$. 
Then by Lemma \ref{StructureLemma}, that is: 
by noting that for $\delta$ small enough $u$ can be approximated by a cylindrical map near $b_{a,R}$, and 
arguing as in the proof of the frequency estimate in Lemma \ref{3DHomogeneous} for general homogeneous minimizers, we have:
$$ N_\phi \left ( z, \gamma R \right ) < \frac{1}{2 \sqrt{\kappa}} + \delta, $$ 
for every $z \in u^{-1} \{ 0 \} \cap B_{\gamma R} \left ( b_{a,R}  \right )$. 

Hence, by Lemma \ref{ConditionalEstimate}:
\begin{equation}
\left | B_\rho \left ( u^{-1} \{0 \} \cap B_{\gamma R} \left ( b_{a,R} \right ) \right ) \right | \leq C \gamma R \rho^2, \quad \forall \rho \in (0, \gamma R],
\label{Iteration1}
\end{equation}
for an absolute constant $C$.
The Hausdorff measure estimate: 
$$ \mathcal{H}^1 \left ( u^{-1} \{ 0 \} \cap B_{\gamma R} \left ( b_{a,R} \right ) \right ) \leq \hat{C} \gamma R, $$ 
follows easily from \eqref{Iteration1}, cf. \cite[Proposition 3.3.3]{Mattila}. 
Moreover, by Step 2 above: $\Gamma_{a,r} \subset B_{\gamma R} \left ( b_{a,R} \right ) \cap u^{-1} \{ 0 \}$. 

For $\lambda$ as in Step 3, we iterate and obtain for $l \geq 1$:
$$ \mathcal{H}^1 \left ( u^{-1} \{ 0 \} \cap B_{\gamma \lambda^l R} \left ( b_{a_\lambda^l,\lambda^l R} \right ) \right ) \leq \hat{C} \gamma \lambda^l R,  $$
and therefore:
\begin{equation}
\mathcal{H}^1 \left ( \Gamma_j \right ) \leq C \cdot \left ( \sum_{l=0}^\infty \lambda^l \right) \cdot \gamma R 
= C \cdot \left ( \frac{3+\gamma}{2\gamma} \right ) \cdot \gamma R = C \cdot \left ( \frac{3+\gamma}{2} \right ) \cdot R.
\label{ArcLength}
\end{equation}
Finally, note that since $\gamma$ depends on $N \left ( 0, 0^+ \right )$ only,
$ \mathcal{H}^1 \left ( \Gamma_j \right ) \leq L $, where $L = L \left ( N \left ( 0, 0^+ \right ), \mathrm{diam}( \mathcal{O}) \right ) $.
\end{proof}

Finally, we are ready to prove the main result of this article.

\begin{proof}[Proof of Theorem \ref{Strengthened}]
Fix $K \subset \subset \Omega$. 
By Corollary \ref{discretezeros}, the set of isolated points in $u^{-1} \{0\} \cap K$ is discrete.
We can cover $u^{-1} \{ 0 \} \cap K$ by neighborhoods $\mathcal{O}$ as in Theorem \ref{LocalFiniteLength} and pass onto a finite subcover.
In particular, we conclude that, $u^{-1} \{ 0 \} \cap K$ consists of a finite union of isolated points and embedded H\"older continuous curves of finite length with finitely many crossings,
and $\mathcal{H}^1 \left ( u^{-1} \{ 0 \} \cap K \right )$ is bounded.
The fact that each H\"older continuous curve in $u^{-1} \{ 0 \} \cap K$ admits a Lipschitz parametrization follows from the $\mathcal{H}^1$-bound, cf. \cite[Theorem I.1.8]{DavidSemmes}.
In conclusion, $u^{-1} \{ 0 \} \cap K$ is rectifiable.
\end{proof}

\section*{Acknowledgment}
I would like to thank Professor Robert Hardt and my thesis advisor Professor Fang-Hua Lin for suggesting this problem and many helpful discussions. 
I would also like to thank the anonymous referee for a careful reading of an earlier version of this article and helpful comments.

\section*{Compliance with Ethical Standards}
The author has no potential conflicts of interest to declare.

\bibliography{bibext}
\bibliographystyle{amsplain}
\end{document}